\documentclass[12pt,english]{amsart}
\calclayout
\usepackage{float}
\usepackage{amsmath,amsfonts,amssymb}
\usepackage{stmaryrd}
\usepackage{graphicx, epsfig, psfrag}
\usepackage{epstopdf}
\usepackage{enumerate}
\usepackage{mathrsfs}
\usepackage{algorithm2e} 
\usepackage{xcolor}
\usepackage{multirow} 
\usepackage{soul}
\usepackage[normalem]{ulem}

\newtheorem{thm}{Theorem}[section]

\theoremstyle{plain}
\newtheorem{theorem}{Theorem}

\newtheorem{remark}[theorem]{Remark}
\newtheorem{Example}{Example}[section]

\newcommand\N{\mathbb N}
\newcommand\R{\mathbb R}

\def\la{{\langle}}
\def\ra{{\rangle}}

\def\bG{{\boldsymbol G}}

\def\btheta{{\boldsymbol\theta}}

\definecolor{bluegreen}{rgb}{0,0.75,0.75}

\def\Ion{{\text{Ion}}}
\def\ion{{\text{ion}}}
\def\Na{{\text{Na}}}
\def\K{{\text{K}}}

\def\rand{{\text{rand}}}

\def\Error{{\text{Error}}}
\def\dist{\operatorname{{dist}}}

\begin{document}

\title[The inverse problem of conductances determination]{A Computational Approach for the inverse problem of neuronal conductances determination}

\author{Jemy A. Mandujano Valle}
\address{Laborat\'orio Nacional de Computa\c c\~ao Cient\'\i fica, Petr\'opolis - RJ, Brazil}
\email{jhimyunac@gmail.com}
\author{Alexandre L. Madureira}
\address{
Laborat\'orio Nacional de Computa\c c\~ao Cient\'\i fica, Petr\'opolis - RJ, Brazil} 
\address{
  EPGE - 
  Funda\c c\~ao Get\'ulio Vargas, Rio de Janeiro - RJ, Brazil} 
\email{alm@lncc.br, alexandre.madureira@fgv.br}
\author{Antonio Leit\~ao}
\address{
Department of Mathematics, Federal University of St. Catarina, P.O. Box 476,
88040-900 Florian\'opolis, Brazil}
\email{acgleitao@gmail.com}
\thanks{The first author would like to thank PCI-CNPq (301330/2020-4) for its financial suport. Also, the second author acknowledges the support of CNPq (grant 307392/2018-0) and FAPERJ (grant E-26/210.162/2019), and the third author acknowledges support from the research agency CNPq (grant 311087/2017-5), and from the AvH Foundation.}

\begin{abstract}
The derivation by Alan Hodgkin and Andrew Huxley of their famous neuronal conductance model relied on experimental data gathered using the squid giant axon. However, the experimental determination of conductances of neurons is difficult, in particular under the presence of spatial and temporal heterogeneities, and it is also reasonable to expect variations between species or even between different types of neurons of the same species.  

We tackle the inverse problem of determining, given voltage data, conductances with non-uniform distribution in the simpler setting of a passive cable equation, both in a single or branched neurons. To do so, we consider the minimal error iteration, a computational technique used to solve inverse problems. We provide several numerical results showing that the method is able to provide reasonable approximations for the conductances, given enough information on the voltages, even for noisy data.

\end{abstract}

\date{May 12, 2020}
\maketitle
\section{Introduction.}
The seminal model of Hodgkin and Huxley~\cite{H-H1952} of neuronal voltage conductance describes how action potential occurs and propagates. It is a landmark model and presents an outstanding combination of modeling based on physical arguments and experimental data, needed to determine the behavior of ion channels. As a part of their work, they modeled the macroscopic behavior of the conductances by designing several mathematical functions (the $\alpha$'s and $\beta$'s) that make the computed voltage to behave as the data. In this paper, we propose a numerical procedure to approximate the conductances of the ion channels, using an iterative method to obtain the unknown parameters. 

Finding the conductances is crucial if one wants to emulate the neuronal voltage propagation using computational models, since the conductances \emph{are part of the data required by the Hodgkin and Huxley model}. Using simpler models might be an alternative, but it is always necessary to find out what are the physiological parameters.  What we would like to offer is a computational way to determine the conductances based on experimental data, and we consider our method a step towards that final goal. The method can also be extended to accommodate excitatory and inhibitory synapses, and could be used, in principle, in several nonlinear models, such as the FitzHugh-Nagumo, Morris-Lecar, Hodgkin-Huxley, etc, with varying degrees of difficulty. We remark that our protocol differs from the experimental setup of Hodgkin-Huxley since voltage-clamp techniques are no longer needed, rendering the method simpler and amenable for use even considering heterogeneous spatially distributed channels. 

We use a simplified neuronal model, the passive cable equation~\cite{bell1990,ermentrout2010,schutter2009}, given by a parabolic partial differential equation (PDE). We consider first the case of a tree with no branches, i.e., possessing only a trunk represented by the interval $[0,L]$. The case of a branched tree is described in Section~\ref{subs2.3}. In the cable model the membrane electrical potential $V:[0,T]\times[0,L]\to\R$ solves   
\begin{equation}\label{equ1}
C_M V_t=\frac {r_a}{2R} V_{xx} -I(t,x)\quad{\text{for }\;\;0\;<t<T,\;\;\, 0<x<L}, 
\end{equation}
where $C_M$ represents membrane specific capacitance in microfarad per square centimeter $(\mu F/cm^2)$; the potential $V$ is in millivolt $(mV)$; the time $t$ is in milliseconds $(ms)$; $r_a$ is the radius of the fiber in millicentimeter $(m cm)$; the specific resistance $R$ is in ohm centimeter $(\Omega cm) $; the space $x$ is in centimeter $(cm)$; the specific ionic current $I$ is in microampere  per square centimeter $(\mu A/cm^2)$.  For the passive cable models, the ionic current is given by
\[
I(t,x)=G_L(V-E_L)+\sum_{i\in\Ion}G_i(t,x)\bigl(V(t,x)-E_i\bigr), 
\]
where the constant leak specific conductance $G_L$ is in millisiemens per square centimeter $(m S/cm^2)$; $E_L$ represents leak equilibrium potential in millivolt $(mV)$; $\Ion$ is the set of ions of the model, e.g., $\Ion=\{\K,\Na\}$. Also, the membrane specific conductance $G_i$ for the ion $i\in\Ion$ is in millisiemens per square centimeter  $(m S/cm^2)$, and it might depend on spatial and temporal variables, as indicated in the notation. Finally, the Nernst potential $E_i$ for each ion $i\in\Ion$ is given in millivolt $(mV)$.

To Eq.~\eqref{equ1} we add boundary and initial conditions given by 
\begin{equation}\label{equ2}
  V_x(t,0)=p(t), \qquad V_x(t,L)=q(t),
  \qquad
 V(0,x)=r(x), \;\;{\text{for}\;\;0<t<T,\;\;0<x<L}.  
\end{equation}
We assume that the constants $C_M$, $r_a$, $R$, $G_L$, $E_L$, $E_i$, $T$ and $L$, and the functions $p$, $q$ and $r$ are given data. 

We next rewrite Eqs.~\eqref{equ1} and~\eqref{equ2} in a slightly more convenient form
\begin{equation}\label{equ3}
\left \{\begin{array}{l}
\displaystyle C_MV_t=\frac{r_a}{2R}V_{xx}-G_L(V-E_L)
- \sum_{i \in Ion}G_i(t,x)[V-E_i],\;\;\;\;  \text{  in  }  (0,T)\times(0,L),
\vspace*{0.2cm}\\
V(0,x)=r(x),\hspace{7.6cm}  \text{  in  }x \in [0,L],
\vspace*{0.2cm}\\ 
V_x(t,0)  =p(t),\;\; V_x(t,L)=q(t),\hspace{4.7cm}  \text{  in  }t \in [0,T]. 
\end{array} \right.
\end{equation}

Let $N_\ion$ be the number of ions of the set \Ion. For $\Ion=\{1,2,\cdots,N_{\ion}\}$, $\bG(t,x)=(G_1(t,x),\dots,G_{N_\ion}(t,x))$. The inverse problem of finding the ``correct'' $\bG$ given measurements of the voltage is highly nontrivial, in the sense that it leads to ill-posed problems~\cite{white1992}, and that it becomes even harder in the presence of \emph{spatially dependent parameters}. There are different approaches to deal with the problem in hand, but certainly no panacea. 

Hodgkin and Huxley~\cite{H-H1952} tackled such problem by a highly nontrivial data fitting, in a wonderful achievement made possible only due to an ingenious combination of experimental and biophysical insight. We refer to~\cite{Bezanilla08} for a quite interesting description of how experimental results evolved through time. From this same reference we quote: \begin{quote}``The simplicity of the squid axon [\dots] was rarely found in other cells such as neurons or cardiac myocytes''\end{quote} and it would be interesting to have a way to combine  experimental data and numerical simulation to unveil the macroscopic conductance behavior even in the absence of such simplicity. Wilfrid Rall and co-authors considered several related questions for the cable equation~\cite{rall1959, rall1960, rall1962, rall1977, rall1992-1, rall1992-2}. See also~\cite{SS98,jack1971, brown1981, durand1983, Aguanno1986, schierwagen1990,kawato1984}. In~\cite{willms1999} there is an interesting attempt to introduce heterogeneity into the Hodgkin and Huxley model. 

In terms of biologically inclined references, the authors of~\cite{alain2017} consider the branched cable equation with the chemical synapses and convert somatic conductances into dendritic conductances. There are several other articles~\cite{bedard2012,kobayashi2011,vich2017,vich2015,yacsar2016} dealing with the issue of determining conductances and pre-synaptic inputs with different techniques, ranging from deterministic to statistical and stochastic. However, it is far from clear if their approach can be mathematically justified and if it is possible at all to extend those ideas for spatially distributed conductances.

We consider next references with a stronger mathematical flavor. The uniqueness of solutions for finding constant parameters in the cable equation, and related methods, are discussed in~\cite{cox1998,cox2000,cox2004}, and~\cite{cox2001-1} for a nonlinear model; see also~\cite{avdonin2013,MR534419} for further considerations related to existence and uniqueness. In~\cite{cox2001-2}, a related problem was tackled based on the FitzHugh--Nagumo and Morris--Lecar models, where nonlinear functions modeling the conductances are sought. The method is based on fixed point arguments, and despite its ingenuity, it is not clear how to extend it to more involved models or to accommodate for spatially distributed ions channels. 

In~\cite{bell2005,tadi2002,cox2006,avdonin2013,avdonin2015}, the question of determining spatially distributed conductances is investigated through different techniques and algorithms. They differ considerably from our method and seem harder to generalize for other situations, as, for instance, when the domain is given by branched trees (with the obvious exception of~\cite{avdonin2013,avdonin2015}), for time-dependent conductances, and for general nonlinear equations, our ultimate goal.

We would like to stress that although neuroscience  models based on ordinary differential equations are, and will always be, of paramount importance, it is our belief that spatially distributed equations will grow in importance. And spatially distributed data will become easier to gather, in particular, due to techniques as \emph{voltage-sensitive dye imaging} (VSDI)~\cite{Casale15555,Grinvald}.

Inverse problems like the present one are ill-posed, and, under certain conditions, the \emph{minimal error method}~\cite{george2017,kaltenbacher2008,neubauer2018} provides a convergent iterative scheme. The main goal of the present paper is to develop such method to approximate the inverse problem of recovering the conductances in the cable equation. We also test the scheme under different scenarios.

The minimal error method is one of several  iterative regularization methods for obtaining stable solutions for ill-posed problems. It has the advantage of each iteration being ``cheap'' (it avoids inversions present in Newton-like methods) at the possible price of taking more iterations to converge. See~\cite{2019arXiv190303130N} for a nice review and comparison among the methods.

We next outline the contents of the paper. In Section~\ref{s:LM}, we present the method, detailing how it should be applied in the cases of non-branched and branched cables, where the geometry is given by a tree. Section~\ref{s:numerics} presents numerical results, and in Section~\ref{s:conc} we draw some concluding remarks. Finally, the Appendix provides some technical details regarding the method and the mathematics behind it. 

\section{Method: The minimal error scheme applied to the conductance determination}\label{s:LM}
We consider here an application of the minimal error method to the problem at hand. Knowing the voltage $V$ at the time-space domain $\Gamma$, we want to determine $\bG$ assuming that Eq.~\eqref{equ3} holds. We consider two different cases in voltage measurement. In the first case, we assume that $V$ is known at all time-space points, i.e., $\Gamma=[0,T]\times[0,L]$ (Table~\ref{tab0}, Case I). In the second case, we assume that the voltage is known only at endpoints and all the time. Thus $\Gamma=[0,T]\times\{0,L\}$ (Table~\ref{tab0}, Case II).
\begin{table}
	\centering \small
	\begin{tabular}{|l|l|l|l|}\hline
CASE I   & $\Gamma=[0,T]\times[0,L]=\{(t,x);\;\;0\leq t\leq T,\;\;0\leq x\leq L \} $   \\\hline
CASE II  & $\Gamma=[0,T]\times\{0,L\} =\{(t,x);\;\;0\leq t\leq T,\;\;x\in\{0,L\}\} $ \\\hline
	\end{tabular}
	\caption{Summary of the two different cases considered in this paper. We seek the unknowns $G_i$ assuming that Eq.~\eqref{equ3} holds and that a measure of the voltage $V$ is known at the space-time domain $\Gamma$ defined above. In case I, the data is known at all points and at all times; in case II, the data is known at two end-points and at all times.}
	\label{tab0}
\end{table}

Let $V|_\Gamma$ be the restriction of $V$ to $\Gamma$, and consider the nonlinear operator
\begin{equation}\label{e:Fdef}
F:D(F)\rightarrow R(F)
\end{equation}
that associates for a given $\bG\in D(F)$ the resulting voltage, i.e., $F(\bG)={V|}_\Gamma$, where $V$ solves Eq.~\eqref{equ3}. The domain $D(F)$ and the image $R(F)=L^2(\Gamma)$ (the space of square integrable functions) are properly defined in the Appendix~\ref{Apendix}. Given a smooth enough function $f$, we define its $L^2$ norm ${\|\cdot\|}_{L^2(\Gamma)}$ such that
\[
\|f\|_{L^2(\Gamma)}^2=\int_{\Gamma}|f(\xi)|^2d\,\xi. 
\]

We consider the inverse problem of finding an approximation for $\bG$ given the noisy data $V^\delta|_\Gamma$, where 
\begin{equation}\label{in4}
{\|V-V^{\delta}\|}_{L^2(\Gamma)}\le\delta,
\end{equation}
for some known noise threshold $\delta>0$. That makes sense since, in practice, the data ${V|}_\Gamma$ are never known exactly. In Section~\ref{s:numerics} we detail the type of noise introduced. 

Given an initial guess $\bG^{1,\delta}$, the minimal error approximation for $\bG$ is defined by the sequence 
\begin{equation}\label{equ6}
\bG^{k+1,\delta}=\bG^{k,\delta}+w^{k,\delta}F'(\bG^{k,\delta})^*(V^\delta|_\Gamma-F(\bG^{k,\delta})),
\end{equation}
for $k=1,2,\dots$, where $F'(\cdot)^*$ is adjoint of the G\^ateux derivative, and 
\begin{equation}
w^{k,\delta}=\frac{{\|V^\delta-F(\bG^{k,\delta})\|}^2_{L^2(\Gamma)}}{{\|F'(\bG^{k,\delta})^*\left(V^\delta-F(\bG^{k,\delta})\right)\|}_{D(F)}^2}.
\end{equation}

The minimal error iteration is a gradient method, as is the steepest descent method. Although the steps of both methods follow the same direction (the gradient), they differ on their step size~\cite{kaltenbacher2008}. As a stopping criteria we use the \emph{discrepancy principle} with $\tau>1$, i.e., we define the stopping iteration step $k_*$ such that 
\begin{equation}\label{equ7}
{\|V^\delta-F(\bG^{k_*,\delta})\|}_{L^2(\Gamma)}
\le \tau\delta
{\le\|V^\delta-F(\bG^{k,\delta})\|}_{L^2(\Gamma)},
\end{equation}
for all $1\le k<k_*$. In practice, stopping criteria are needed for all iterative methods, otherwise the scheme might stop before it is accurate enough, or diverge, or even waste computing time without significantly improving the solution. For inverse problems with noisy data this is even more crucial since running regularization iterative methods beyond certain threshold forces the method to ``fit the noise''. It is possible to show that, under certain conditions (we assume that is the case), $\bG^{ {k_*},\delta}$ converges to a solution of $F(\bG)=V$ as $\delta\to0$; see~\cite[Theorem 2.6]{kaltenbacher2008}.

{From Eqs.~\eqref{equ6} and~\eqref{equ7} we obtain an approximation $\bG^{k_*,\delta}$ for $\bG$. Although the adjoint $F'(\bG^{k,\delta})^*$ is not known, it is possible to show that Eq.~\eqref{equ6} is actually } 
\begin{equation}\label{equ16}
G_i^{k+1,\delta}(t,x)=G_i^{k,\delta}(t,x)- w^{k,\delta}(V^{k,\delta}(t,x)-E_i)U^k(t,x)\quad\text{for all }i\in\Ion,
\end{equation}
where
\begin{equation*}
w^{k,\delta}=\frac{{\|V^\delta-F(\bG^{k,\delta})\|}^2_{L^2(\Gamma)}}{\displaystyle \sum_{i\in \Ion}{ \big\|\left(V^{k,\delta}(t,x)-E_i\right)U^k(t,x)\big\|}^2_{D(F)}}.
\end{equation*}
Also $V^{k,\delta}$  solves Eq.~\eqref{equ3} with $\bG$ replaced by $\bG ^{k,\delta}$, and $U^k$ solves the following PDE with \emph{final condition}:
\begin{equation}\label{equ10}
\left \{\begin{array}{l}
\displaystyle
-\frac{r_a}{2R}U_{xx}-C_M U_t+G_LU\vspace*{0.1cm}+\sum_{i\in\Ion}G_i(t,x)U=\alpha_1 \left(V^\delta-V\right),\;\;\;\;  \text{  in  }  (0,T)\times(0,L), \\
U(T,x)=0,\hspace{8.75cm}  \text{  in  }x \in [0,L],
\vspace*{0.3cm}\\ 
\displaystyle U_x(t,0)=-\alpha_2\frac{2R}{r_a}\left(V^\delta(t,0)-V(t,0)\right),\hspace{4.0cm}  \text{  in  }t \in [0,T],\\
\displaystyle U_x(t,L)=\alpha_2\frac{2R}{r_a}\left(V^\delta(t,L)-V(t,L)\right),\hspace{4.15cm}  \text{  in  }t \in [0,T].
\end{array} \right.
\end{equation}

The constants $\alpha_1$, $\alpha_2$ depend on the set $\Gamma$ as follows: 
\begin{equation}\label{e:alphadef}
(\alpha_1,\alpha_2)=\begin{cases}(1,0)&\text{if } \Gamma=[0,T]\times [0,L],
\\
(0,1)&\text{if }\Gamma=[0,T]\times \{0,L\}.\end{cases}
\end{equation}

In Theorem~\ref{t:main} of the Appendix, we show how to obtain Eq.~\eqref{equ16} from Eq.~\eqref{equ6}.

\begin{remark}\label{remark1}
Note from Eq.~\eqref{equ16} that $G_i^{k+1,\delta}(T,x)=G_i^{k,\delta}(T,x)$ for all $x\in[0,L]$ and every $k\in\N$, since, from Eq.~\eqref{equ10}, $U^k(T,x)=0$. Thus, $G_i^{k,\delta}$ is \emph{never corrected at the final time $T$}. To recover $G_i$ at time $T$, one could perform the computations up to $2T$, and then consider only the values up to $T$.
\end{remark}

The numerical scheme of our method is as follows. Check Table~\ref{tab0} for notation. Note from Algorithm~\ref{a:Land} that solutions of two PDEs are needed for each iteration. 

\begin{algorithm}[H]\label{a:Land}
\KwData{$V^\delta|_{\Gamma}$, $r$, $p$, $q$, $\delta$, $\tau$}
\KwResult{Compute an approximation for $\bG$ using minimal error Iteration Scheme}
Choose $\bG^{1,\delta}$ as an initial approximation for $\bG$\;
Compute $V^{1,\delta}$ from Eq.~\eqref{equ3}, replacing $\bG$ by $\bG^{1,\delta}$\;
k=1\;
\While{$\tau\delta\le{\|V^\delta-V^{k,\delta}\|}_{L^2(\Gamma)}$}{
  Compute $U^k$ from Eq.~\eqref{equ10}\;
  Compute $\bG^{k+1,\delta}$ using Eq.~\eqref{equ16}\;
  Compute $V^{k+1,\delta}$ from Eq.~\eqref{equ3}, replacing $\bG$ by $\bG^{k+1,\delta}$\;
  $k\gets k+1$\;
  }
 \caption{Minimal Error Iteration}
\end{algorithm}


\begin{remark}\label{remark2}
Whenever $\bG$ is time independent, and in this case we write $\bG(t,x)=\bG(x)$, the iteration is defined by 
\begin{equation}\label{equ17}
G_i^{k+1,\delta}(x)=G_i^{k,\delta}(x)-w^{k,\delta} \frac1T \int_0^T(V^{k,\delta}(t,x)-E_i)U^k(t,x)\,dt
\quad\text{for } i\in\Ion,
\end{equation}
where
\begin{equation*}
w^{k,\delta}=\frac{{\|V^\delta-F(\bG^{k,\delta})\|}^2_{L^2(\Gamma)}}{ \displaystyle\sum_{i\in\Ion}{\left\|\displaystyle\frac1T\int_0^T\left(V^{k,\delta}(t,x)-E_i\right)U^k(t,x)dt\right\|}^2_{ L^{\infty}(0,L)} }. 
\end{equation*}
\end{remark}

\subsection{The minimal error method applied to the conductance determination  defined on a tree}\label{subs2.3}
We now describe the equivalent formulation when the domain is given by a tree, and we consider for simplicity a tree with a single branch. Note however that the generalizations for more complex geometries are straightforward. In general, the tree $\Theta=\mathcal E\cup\mathcal V$ is defined by the set of edges $\mathcal E$ and the set of vertices $\mathcal V$. Boundary conditions are imposed on the terminal vertices $\mathcal V_T$, and transmission conditions are imposed on the interior vertices $\mathcal V_I$. Of course, $\mathcal V_T\cup\mathcal V_I=\mathcal V$. In the case depicted in Figure~\ref{f:tree},
\[
\mathcal E=\{e_1,e_2,e_3\},\quad\mathcal V =\{\nu_1,\nu_2,\nu_3,\nu_4\},\quad\mathcal V_T=\{\nu_1,\nu_3,\nu_4\}=\{\gamma_1,\gamma_2,\gamma_3\},\quad\mathcal V_I=\{\nu_2\}.
\]

\begin{figure}[h!]
	\includegraphics[ height=5cm, width=14.cm]{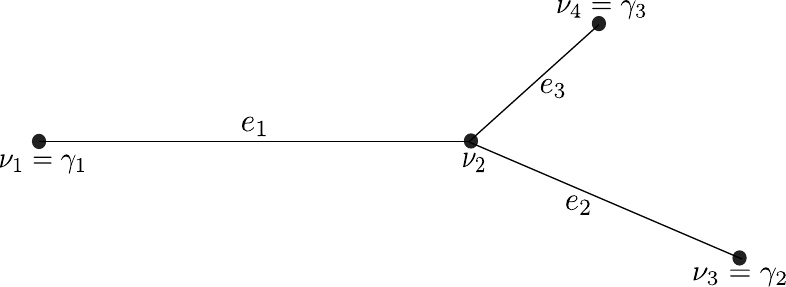}
	\caption{Example of a branched tree with one bifurcation point.}
	\label{f:tree}
\end{figure}

Our cable equation model defined on a branched tree is given by 
\begin{equation}\label{equ20}
\left \{\begin{array}{l}
\displaystyle C_MV_t=\frac{r_a}{2R}V_{xx}-G_L(V-E_L)- \sum_{i \in Ion} G_i(t,x)\left[V-E_i\right], \hspace*{0.3cm}\text{in}\;(0,T)\times\mathcal{E},
\vspace*{0.2cm}
\\
V(0,x)=r(x),\hspace*{7.5cm}\text{  in   }x \in  \Theta,
\vspace*{0.1cm}
\\
V_x(t,\gamma_1)=p(t),\hspace*{7.4cm}\text{  in  }t \in [0,T],
\vspace*{0.1cm}
\\
V_x(t,\gamma_2)=V_x(t,\gamma_3)=q(t),\hspace*{5.4cm}\text{  in }t \in [0,T],
\vspace*{0.2cm}
\\
{V}_x^{e_1}(t,\nu_2)-{V}_x^{e_2}(t,\nu_2)-{V}_x^{e_3}(t,\nu_2)=0,\hspace*{3.4cm}\text{  in }t \in [0,T],
\end{array} \right.
\end{equation}
where $V_x^{e_j}(t,\nu_2)$ denotes the derivative of $V$ at the vertex $\nu_2$ taken along the edge $e_j\in\{e_1,e_2,e_3\}$ in the direction towards the vertex.
%

Consider the operator~\eqref{e:Fdef}, where $F(\bG)=V(\cdot,\cdot)$ and $V$ solves Eq.~\eqref{equ20}. Given $V^\delta$, our goal is to obtain an approximation to $\bG$ using the iteration~\eqref{equ6}. To compute the adjoint operator $F^{'}(\cdot)^*$, we define the following PDE:
\begin{equation}\label{equ21}
\left \{\begin{array}{l}
\displaystyle-\frac{r_a}{2R}U_{xx}^k-C_MU_t^k+G_LU^k +\sum_{i \in Ion} G_i^{k,\delta}(t,x)U^k\displaystyle=V^\delta-V^{k,\delta},\hspace*{0.5cm}\text{in }(0,T)\times\mathcal{E},
\vspace*{0.2cm}
\\
U^k(T,x)=0,\hspace*{8.7cm}\text{  in   }x \in  \Theta,
\vspace*{0.2cm}
\\
U^k_x(t,0)=U^k_x(t,\nu_2)=U^k_x(t,\nu_3)=0,\hspace*{4.7cm} \text{  in }t \in [0,T],
\vspace*{0.2cm}
\\
{U^k_x}^{e_1}(t,\nu_1)-{U^k_x}^{e_2}(t,\nu_1)-{U^k_x}^{e_3}(t,\nu_1)=0,\hspace*{3.75cm}\text{  in } t \in [0,T].

\end{array} \right.
\end{equation}

We then compute $G_i^{k+1,\delta}$ according to~\eqref{equ16}. Remarks \ref{remark1}--\ref{remark2} also hold for this problem. 

\section{Results: Numerical Simulation}\label{s:numerics}
In this section we test the method under different scenarios. Of course, the solutions are obtained numerically, and for that we use finite difference scheme in space coupled with backward Euler in time. To compute the integral in Eq.~\eqref{equ17} we use the trapezoidal rule. In what follows we assume that the numerical approximations are accurate enough. All the experiments were performed using Matlab\textsuperscript{\textregistered}, and the codes are available at \cite{Mandujano2020}. 

To design our \emph{in silico} experiments, we first choose $\bG$ and compute $V$ from Eq.~\eqref{equ3}, obtaining then ${V|}_\Gamma$. Of course, in practice, 
{only} the  values of ${V^\delta|}_\Gamma$ are given by some experimental measures, and thus subject to experimental/measurement errors. In our examples, $V^\delta|_{\Gamma}$ is obtained by considering additive-multiplicative noise
\begin{equation}\label{equ22}
  {
    V^\delta(t,x)=V(t,x)+(aV+b)\rand_\Delta(t,x) \;\;\;\text{for all } (t,x) \in \Gamma,}
\end{equation}
{for scalars $a$, $b$, and $\rand_\Delta$ is a uniformly distributed random variable taking values in the range $[-\Delta,\Delta]$. The threshold $\delta$ is such that (cf. Eq.~\eqref{in4}) ${\|(aV+b)\rand_\Delta\|}_{L^2(\Gamma)}\le\delta$, and we impose then}
\begin{equation}\label{e:mn}
{{\|(aV+b)\|}_{L^2(\Gamma)}\Delta=\delta. }
\end{equation}

In our numerical examples, we use multiplicative and additive noises, i.e., $a=1/2$ and $b=1/2$ at Eq.~\eqref{equ22}. We noticed no qualitative difference in the results for other values of $a$ and $b$. 
\begin{remark}
It is not possible in general to predict how the added noise will affect the conductances since the operator $F$ defined in Eq.~\eqref{e:Fdef} is not bounded, meaning that small perturbation of the data might lead to large perturbation of the conductances. That is why inverse problems are so hard to approximate.  
\end{remark}


Next, given the initial guess $\bG^{1,\delta}$, the data ${V^\delta|}_{\Gamma}$, and the noise threshold $\delta$, we approximate $\bG$ using the Algorithm~\ref{a:Land}. Unlike in ``direct'' PDE problems where the exact solution usually has to be computed by numerical over-kill, here we have the exact $\bG$ and we use that to gauge the algorithm performance.

In our numerical examples we consider $M=50$ experiments for each $\delta$. Therefore, for an experiment $j$, we measured voltage ${{V^\delta|}_{\Gamma}}_j$ and compute the approximation $\bG^{k_*,\delta}_j$ of unknown conductance $\bG$. We define the mean and standard deviation for a function $f$ (in what follows, $f$ will be either $V^\delta|_\Gamma$ or $\bG^{k_*,\delta}$): 
\begin{equation}\label{equa18}
\mu_f(t,x)=\frac{1}{M}\displaystyle\sum_{j =1}^{M} f_j(t,x),
\qquad
\sigma_f(t,x)=\sqrt{\frac{1}{M}\displaystyle\sum_{j =1}^{M} (f_j(t,x)-\mu_f(t,x))^2 }, 
\end{equation}
if $f=\bG^{k_*,\delta}$ and $\Ion=\{1,2,\cdots,N_{\ion}\}$ then $\mu_f=\left(\mu_{G_1^{k_*,\delta}},\mu_{G_2^{k_*,\delta}},\cdots,\mu_{G_{\N_{\ion}}^{k_*,\delta}}\right)$.
	
From \eqref{equ7}, for iteration $k^*$, we introduce 
\begin{equation}\label{equ23}
\qquad
{
	\Error_{\bG}=\frac1{N_{\ion}}\sum_{i\in Ion}\int_{D(\bG)}\frac{|G_i-\mu_{G_i^{k_*,\delta}}|} {|G_i|}\times100\%, }
\end{equation}
where $D(\bG)$ denotes the domain where $\bG$ is defined. Given value $V$ and its approximation  $\mu_{{V^\delta|}_{\Gamma} }$, we define the error 
\begin{equation}\label{equ18}
\Error_{V}=\int_{\Gamma}\frac{\big|V(t,x)-\mu_{{V^\delta|}_{\Gamma} }(t,x) \big|}{|V(t,x)|}d\Gamma \times 100\%.
\end{equation}
\begin{remark}
There seems to be no common agreement in the literature on what would be the ``best'' error measurement, and the choice of the above \emph{Mean Absolute Percentage} error is fully justified in, e.g.,~\cite{DEMYTTENAERE201638}. In practice, after discretizing the equations and the unknown functions, only nodal values are known. Consider the space-time discretization $t_n=(n-1)T/(N-1)$ for $n=1,2\cdots,N$ and $x_j=(j-1)L/(J-1)$ for $j=1,2\cdots,J$. Thus, the relative error introduced above~\eqref{equ23} is approximated by 
\begin{equation}\label{equa20}
{\Error}_{\bG}=\frac1{N_{\ion}}\frac TN\frac LJ
\sum_{i\in\Ion}\sum_{n=1}^{J}\sum_{j=1}^N\left|{\frac{G_i(t_n,x_j)-G^{k,\delta}_i(t_n,x_j)}{G_i(t_n,x_j) }}\right|\times 100\%.
\end{equation}

Whenever $\bG$ is time independent, and in this case we write $\bG(t_n,x_j)=\bG(x_j)$, the mean absolute percentage error is defined by
\begin{equation}\label{equa21}
{\Error}_{\bG}=\frac{1}{N_{\ion}}\frac LJ
\sum_{i\in\Ion}\sum_{n=1}^{J}\left|{\frac{G_i(x_j)-G^{k,\delta}_i(x_j)}{G_i(x_j) }}\right|\times 100\%.
\end{equation}

Similar remark holds for other norms, e.g., for $\Gamma=[0,T]\times[0,L]$, ${\|f\|}_{L^2(\Gamma)}$ is to be replaced by ${\|f\|}_{l^2(\Gamma)}$, where 
\begin{equation}\label{e:l2}
{\|f\|}_{l^2(\Gamma)}^2
=\frac TN\frac LJ\sum_{(t_n,x_j)\in \Gamma} {|f (t_n,x_j)|}^2.
\end{equation}
\end{remark}

\begin{remark}
	\emph{The mean absolute percentage error}, for  \eqref{equ18} and $\Gamma=[0,T]\times[0,L]$, is
\begin{equation}\label{equa23}
{\Error}_{V}=\frac TN\frac LJ
\sum_{n=1}^{J}\sum_{j=1}^N\left|{\frac{V(t_n,x_j)-\mu_{{V^\delta|}_{\Gamma} } (t_n,x_j)}{V(t_n,x_j) }}\right|\times 100\%,
\end{equation}
and, for \eqref{equ18} and $\Gamma=[0,T]\times\{0,L\}$, is
\begin{equation}\label{equa24}
{\Error}_{V}=\frac 12\frac TN\sum_{j=1}^N\left[\left|{\frac{V(t_n,1)-\mu_{{V^\delta|}_{\Gamma} } (t_n,1)}{V(t_n,1) }}\right|+\left|{\frac{V(t_n,J)-\mu_{{V^\delta|}_{\Gamma} } (t_n,J)}{V(t_n,J) }}\right|\right]\times 100\%.
\end{equation}
\end{remark}

We present four numerical tests. In the first three examples the geometry is defined by a segment, and in the fourth example it is given by a branched tree. The first example considers only one ion $(\Ion=\{\K\})$, with $\bG(x)=G_{\K}(x)$ depending only on the spatial variable, and the voltage is known at $\Gamma=[0,T]\times\{0,L\}$, i.e., at all times but only at the endpoints. In the second example, still with one ion $(\Ion=\{\K\})$, the conductance depends on the temporal and spatial variables $(t,x)$ and measured voltage is known at $\Gamma=[0,T]\times[0,L]$, i.e., all the time and at all points. In the third example, we consider two ions $(\Ion = \{\K,\Na\})$, where $\bG(x)=(G_\K(x), G_\Na(x))$ depends only on the spatial variable and the data is again known at $\Gamma=[0,T]\times[0,L]$, i.e., all the time for all points. Finally, in the fourth example we consider the case where the geometry is defined by a tree, with the conductance being time independent under the presence of one ion, and the voltage data being known at all the time and all the points. 

For all numerical examples let $C_M=1\;(\mu F/cm^2)$, $r_a=0.0238 \;(cm)$, $R=34.5 \;(\Omega cm)$, $G_L=0.3\; (m S/cm^2)$ and $E_L=10.613\;(mV)$; see \cite[page 69]{tuckwell1988}, \cite[page 4]{kashef1974} and \cite[page 586]{cooley1966}.  The following Neumann boundary conditions
\begin{equation*}
V_x(t,0)=p(t)=-\frac{RI(t)}{\pi {r_a}^2},\;\;\;V_x(t,L)=q(t)=0,
\end{equation*}
correspond to injecting a current $I(t)$ at $x=0$, and imposing a sealed end condition at $x=L$; see~\cite{cox2006,cox2001-1,mascagni1990,mascagni1989}. 
We consider the current $I(t)=0.1t^2 \exp(-10t)\;\; (m A)$ and time $T=20 \;(ms)$, see~\cite[page 101]{cox2001-1}.   We assume that initial condition $V(0,x)=r(x)=0\; (mV)$, see~\cite{cox2006,cox2001-1}. As previously mentioned, we consider $M=50$ experiment for each additive-multiplicative noise~\eqref{equ22}, where we set $\Delta=25\%$, $5\%$, $1\%$ and $0.2\%$~\cite{bell2005,cox2006,tadi2002}. For stopping criterion~\eqref{equ7}, we choose $\tau=1.01$. Finally, we set the initial guesses to be zero in the iterative scheme, and thus the initial relative error is $100 \%$.

In Examples~\ref{Exa3.1},~\ref{Exa3.2} and~\ref{Exa3.3}, we consider $L=0.1\;(cm)$, see \cite[page 145]{cox2006}. We solve equations~\eqref{equ3} and~\eqref{equ10} via finite differences in space, with $\Delta x=0.001\;(cm)$, and in time, with $\Delta t=0.2\;(ms)$. 

For Example~\ref{Exa3.4}, the length of the edges are: $|e_1|=|e_2|=0.1\; (cm)$ and $|e_3|=0.2\;(cm)$. The values of the other parameters are the same as above. In this example, we solve the differential equations~\eqref{equ20} and~\eqref{equ21} again using finite differences, but with explicit Euler in time, and $\Delta x=\Delta t=0.01$.

\Example \label{Exa3.1}
Consider a particular instance from Eq.~\eqref{equ3}, where $N_\ion=1$ $(\Ion=\{\K\})$, $E_\K=-12\;(mV)$ and  $G_i(t,x)=G_\K(x)$.  The goal is to estimate 
$$G_{\K}(x)=0.2+0.2/\left(\;1+\exp(\;(\;0.1/2-x\;)/0.01\;)\;\right)\;(mS/cm^2)$$
 given $V^\delta|_{\Gamma}=\big\{\bigl(V^\delta(t,0),\;V^\delta(t,L)\bigr):\,t\in[0,20]\big\}$. See~\cite{cox2006}. 

In Table~\ref{t:ex1} we present the results for various levels of noise.  At each line of the table the noise is reduced by a factor of five, and that leads to a similar reduction of the residual. The same cannot be stated about the approximation error, exposing the instability of the problem. 

In Figures~\ref{Ex1-fig1} and $\ref{Ex1-fig2}$, we plot results for $\Delta=1\%$ of noise  with  $M=50$ experiments (see Table~\ref{t:ex1}, line 4). In Figure~\ref{Ex1-fig1}, we display the exact membrane potential, the mean and  standard deviation of the voltage measurement (see Eq.~\eqref{equa18}, for $f={V^\delta|}_{\Gamma}$),  and the difference between the exact membrane potential and mean of its measurements. In Figure~\ref{Ex1-fig2}, we show the exact conductance, the mean and standard deviation of the approximate solutions (see equation \eqref{equa18}, for $f=\bG^{k_*,\delta}$), and the difference between the conductance and  mean of the approximate solutions. We also consider the higher noise level $\Delta=5\%$ and present the related figures in the Online Resource. 


\begin{table}[h!]
\small
\begin{tabular}{|l|l|l|c|c|c|c|c|c|c|}\hline
		$\Delta$  & ${\Error}_{\bG}$ & ${\Error}_V$ \\\hline  
$25\%$    & 2.0387 \%        & 26.4003\%    \\\hline
$5\%$     & 0.7738 \%        & 4.7587 \%    \\\hline
$1\%$     & 0.3306 \%        & 0.9567 \%    \\\hline
$0.2\%$   & 0.2034 \%        & 0.1893 \%    \\\hline
\end{tabular}
\caption{\footnotesize {Numerical results for Example~\ref{Exa3.1} with $M=50$ experiments for each noise level $\Delta$. The first column describes the noise level $\Delta$, as in  Eq.~\eqref{equ22}. The second column contains the mean errors according to Eq.~\eqref{equa21}. Finally, the third column contains the mean errors according to Eq. ~\eqref{equa24}.
}}
	\label{t:ex1}
\end{table}
\begin{figure}
	\centering
	\includegraphics[ height=5cm, width=14.cm]{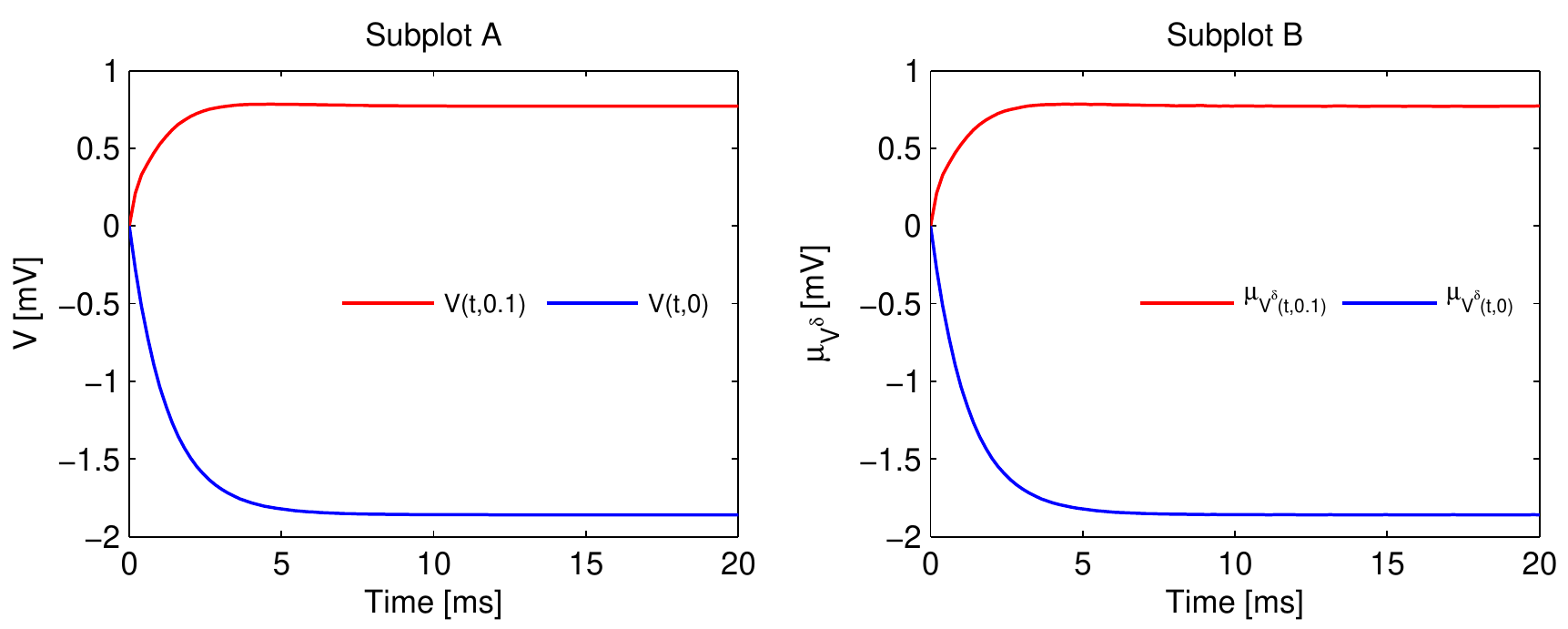}
	\includegraphics[ height=5cm, width=14.cm]{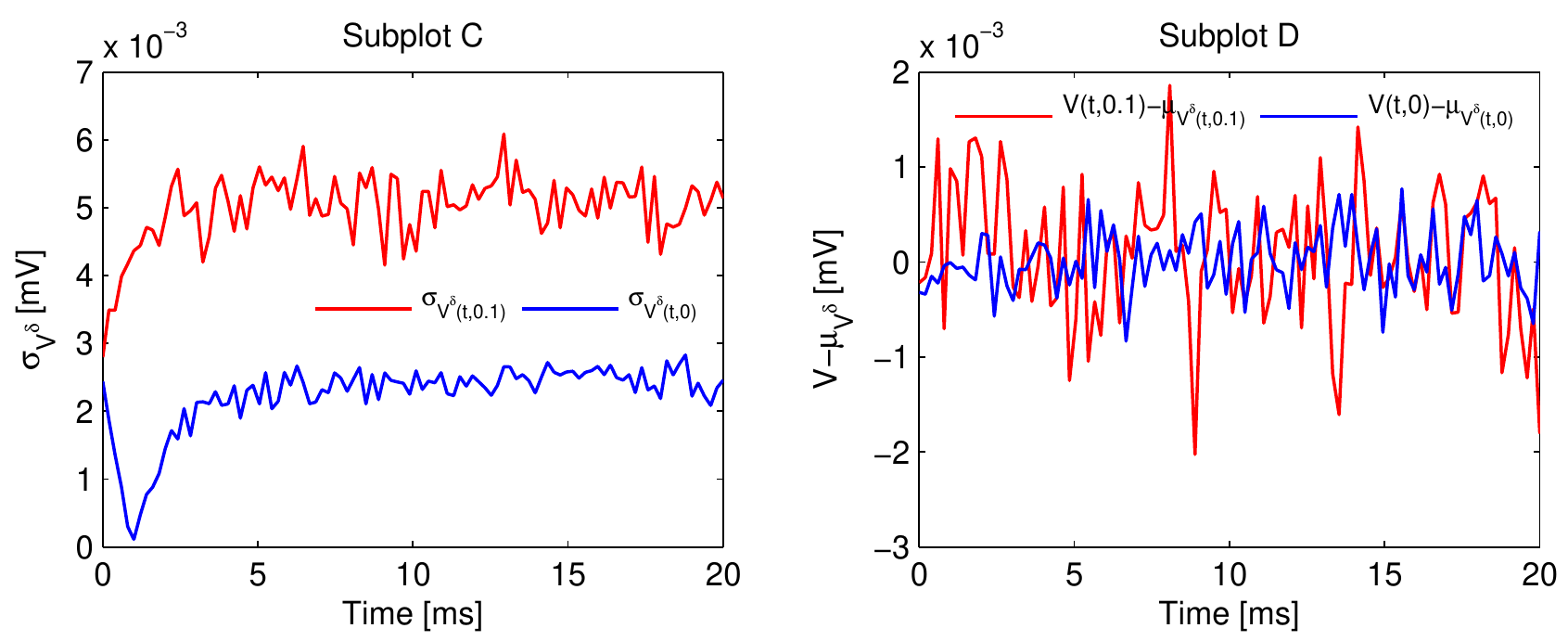}
	\caption{\footnotesize {Result for Example~\ref{Exa3.1} with $\Delta=1\%$ noise. Subplot A presents the exact, deterministic, membrane potential at the end-points. Subplots B and C show the mean and standard deviation of the one hundred membrane potential measurements, respectively. In Subplot D displays the difference between the exact membrane potential and the mean of its measurements.} } 
	\label{Ex1-fig1}
\end{figure}

\begin{figure}
	\centering
	\includegraphics[ height=5cm, width=14.cm]{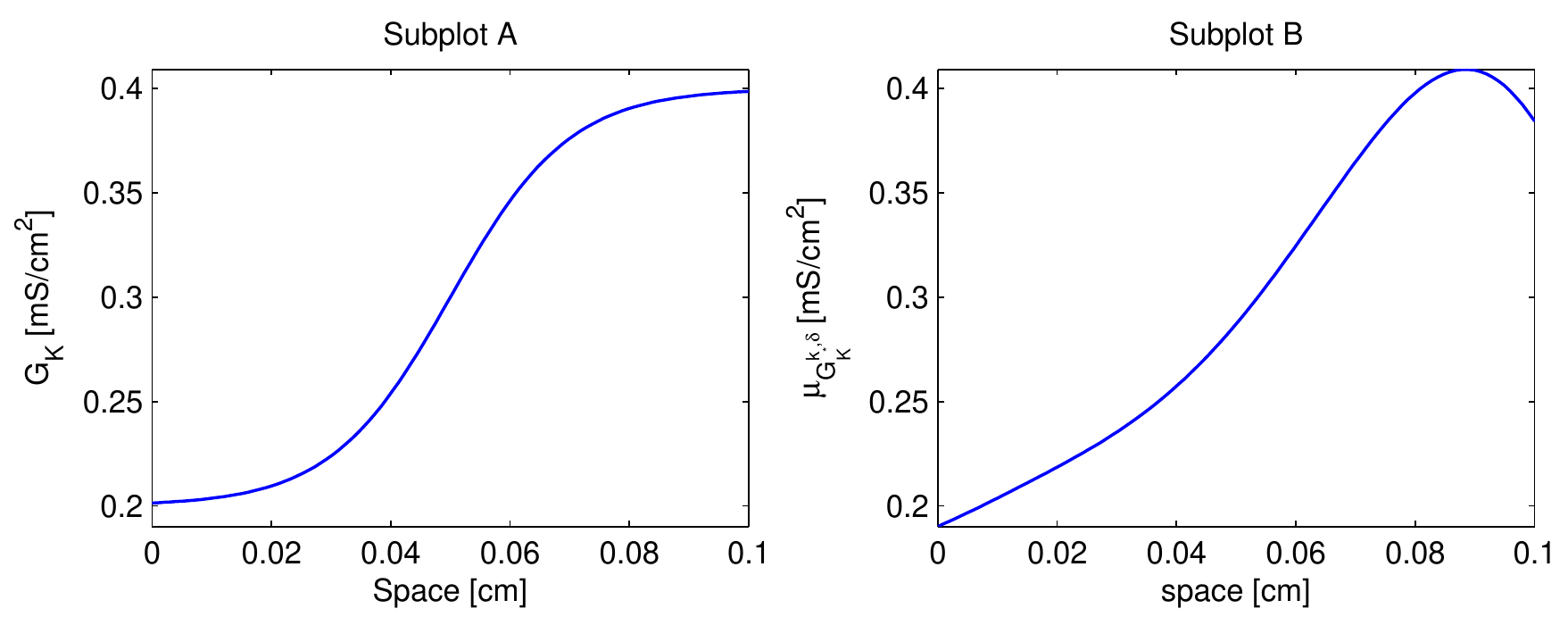}
	\includegraphics[ height=5cm, width=14.cm]{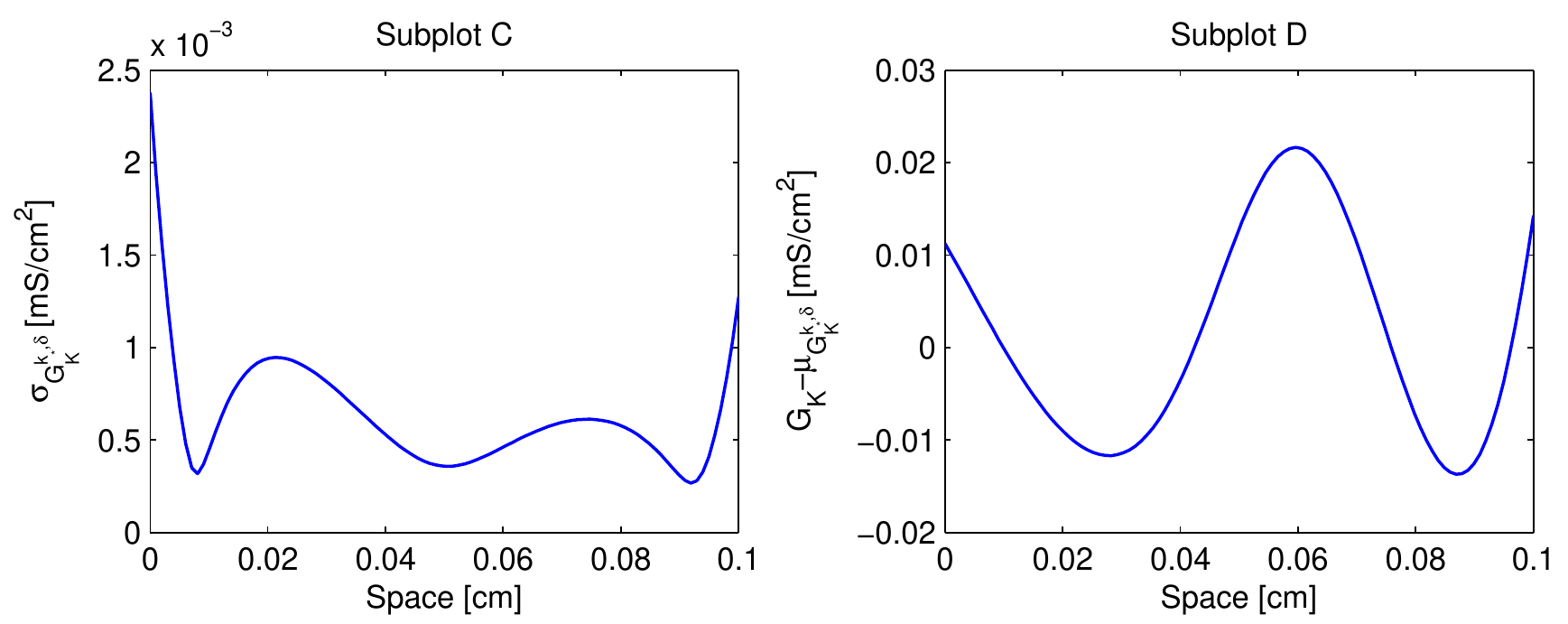}
	\caption{{\footnotesize For Example~\ref{Exa3.1} with $\Delta=1\%$. Subplot A shows the exact solution $G_{\K}$. In subplots B and C present the mean and standard deviation of the one hundred approximate solutions, respectively. Subplot D displays the difference between the exact solution and mean of the approximate solutions.} }
	\label{Ex1-fig2}
\end{figure}

\Example\label{Exa3.2} We consider conductance as depending on time and space, where $N_\ion=1$ $(\Ion=\{\K\})$,  $E_{\K}=-12\;(mV)$, $G_i(t,x)=G_\K(t,x)$.  The goal is to estimate 
\[
G_{\K}(t,x)=0.2+0.2/\left(\;1+\exp(\;(\;0.1/2-x\;)/0.01\;)\;\right)+t+1 \;(mS/cm^2)
\]
given $V^\delta|_{\Gamma}=\big\{V^\delta(t,x):\,(t,x)\in [0,20]\times[0,0.1]\big\}$.

This example is harder than the previous one since now the conductance depends on both time and space. In Table~\ref{t:ex2} we present the results for various levels of noise, and the same comments of Example~\ref{Exa3.1} apply. In Figures \ref{Ex2-fig1} and \ref{Ex2-fig2}, we plot numerical results for $\Delta=1\%$ of noise  with $M=50$ experiments (see Table~\ref{t:ex2}, line 4). Observe that the data for both $V^\delta|_{\Gamma}$ and $G_K$ depend on time and space. See also our Online Resource for experiments with $\Delta=5\%$.

\begin{table}[h!]
	\small
\begin{tabular}{|l|l|l|c|c|c|c|c|c|c|}\hline
	$\Delta$  & $\Error_\bG$  & $\Error_{V}$  \\\hline  
$25\%$    & 17.7328 \%    & 4.2053 \%     \\\hline
$5\%$     &  2.3030 \%    & 0.6823 \%     \\\hline
$1\%$     &  0.5655 \%    & 0.1312 \%     \\\hline
$0.2\%$   &  0.2061 \%    & 0.0309 \%     \\\hline
\end{tabular}

	\caption{\footnotesize {Numerical results for Example~\ref{Exa3.2} with $M=50$ experiments for each noise level $\Delta$. The first column describes the noise level $\Delta$, as in  Eq.~\eqref{equ22}. The second column contains the mean errors according to Eq.~\eqref{equa20}. Finally, the third column contains the mean errors according to Eq. ~\eqref{equa23}.}}
	\label{t:ex2}
\end{table}

\begin{figure}
	\centering
	\includegraphics[ height=5cm, width=14.cm]{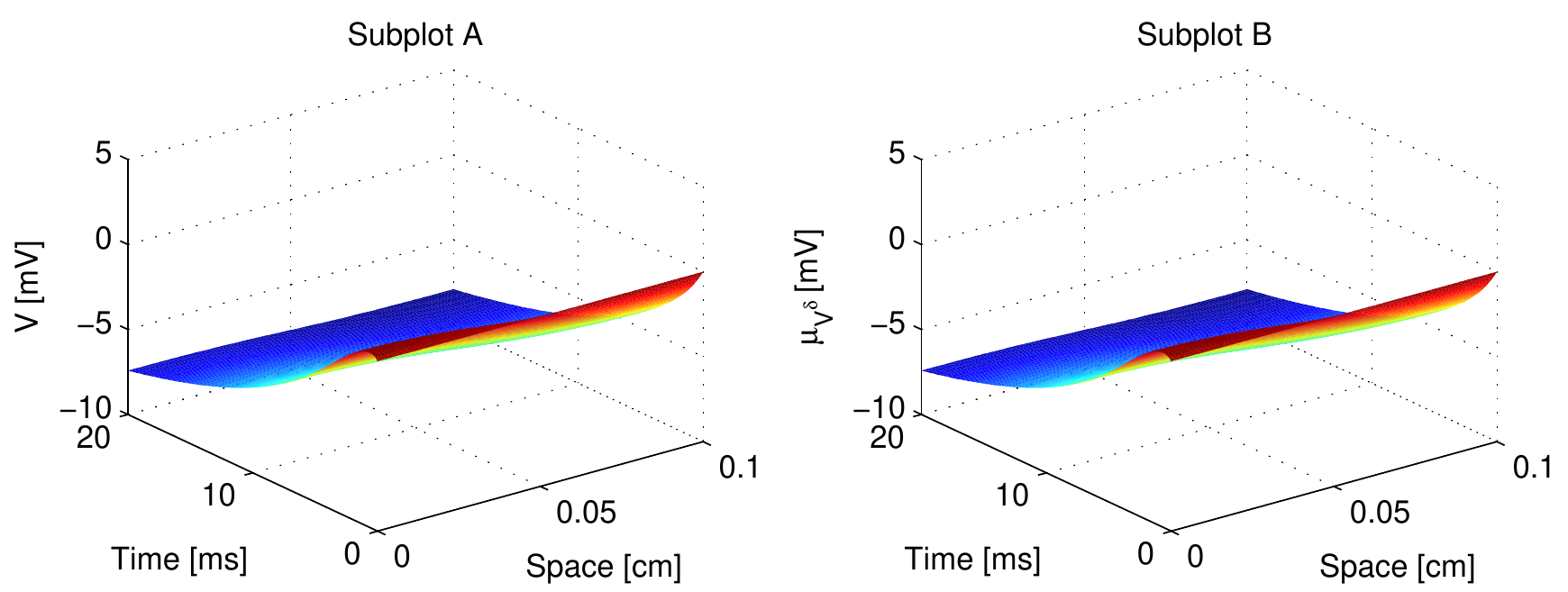}
	\includegraphics[ height=5cm, width=14.cm]{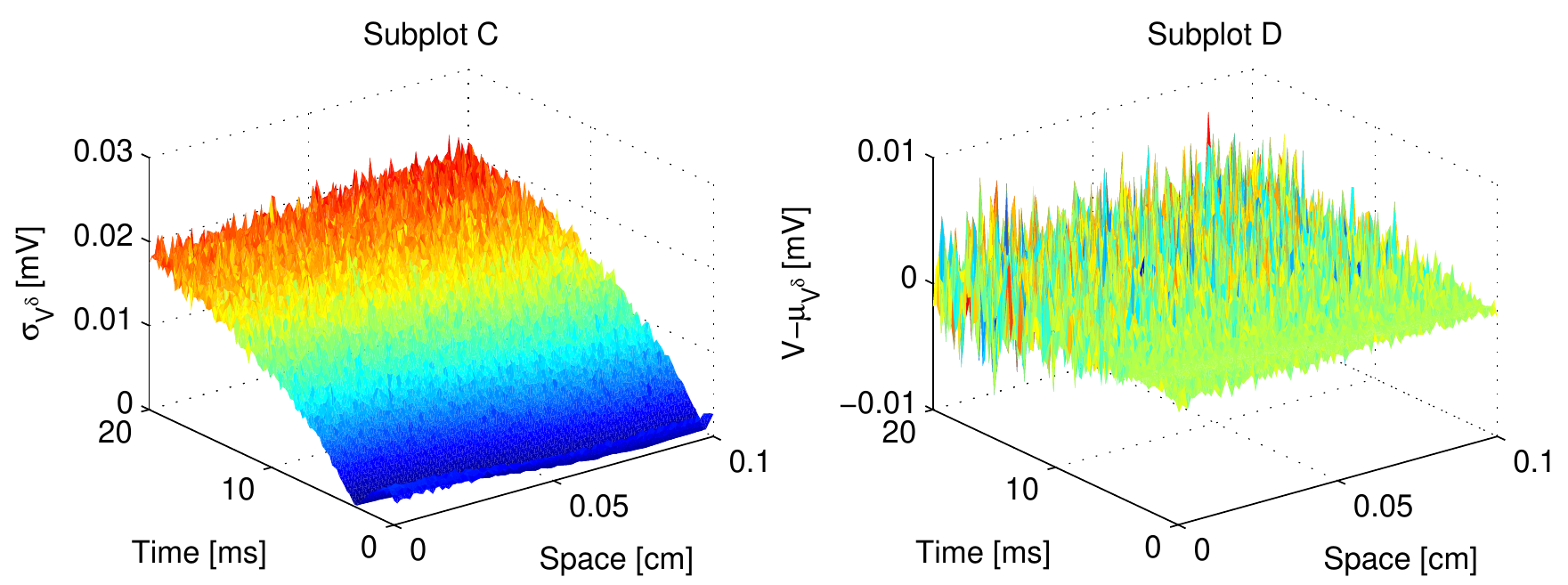}
	\caption{{\footnotesize For Example \ref{Exa3.2} with  $\Delta=1\%$. See Figure~\ref{Ex1-fig1} for the   subplots description.}}
	\label{Ex2-fig1}
\end{figure}   
\begin{figure}
	\centering
	\includegraphics[ height=5cm, width=14.cm]{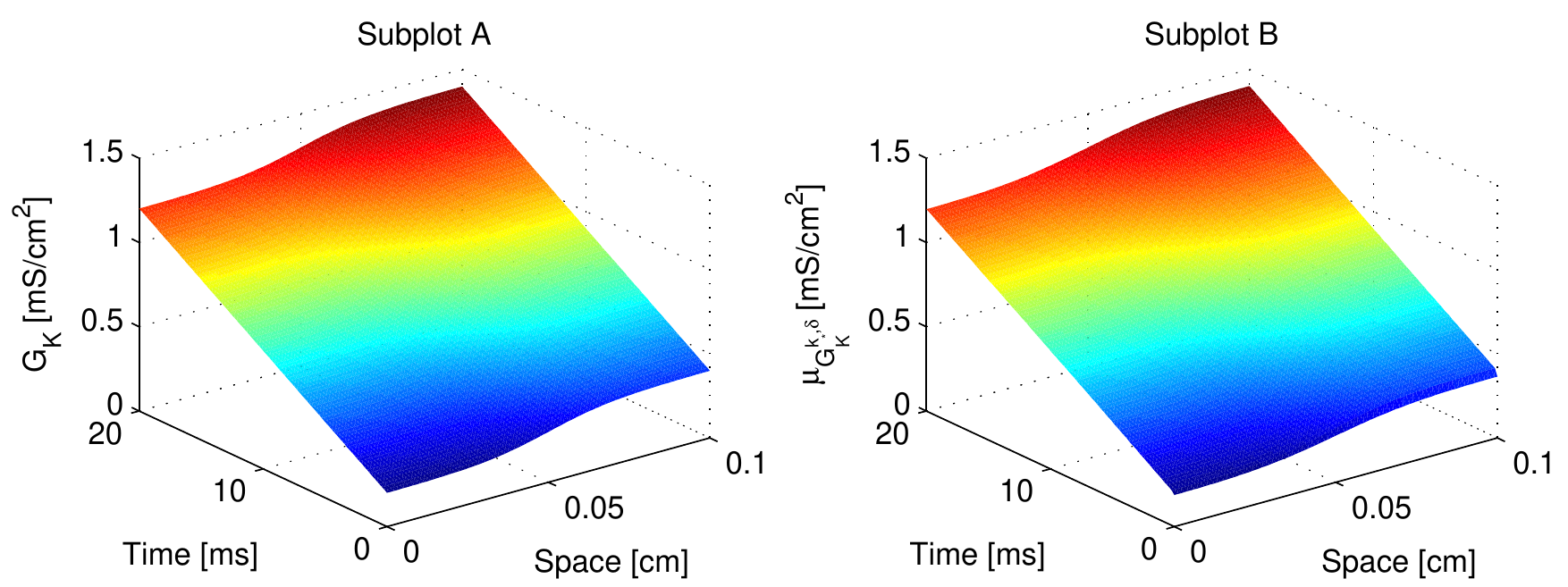}
	\includegraphics[ height=5cm, width=14.cm]{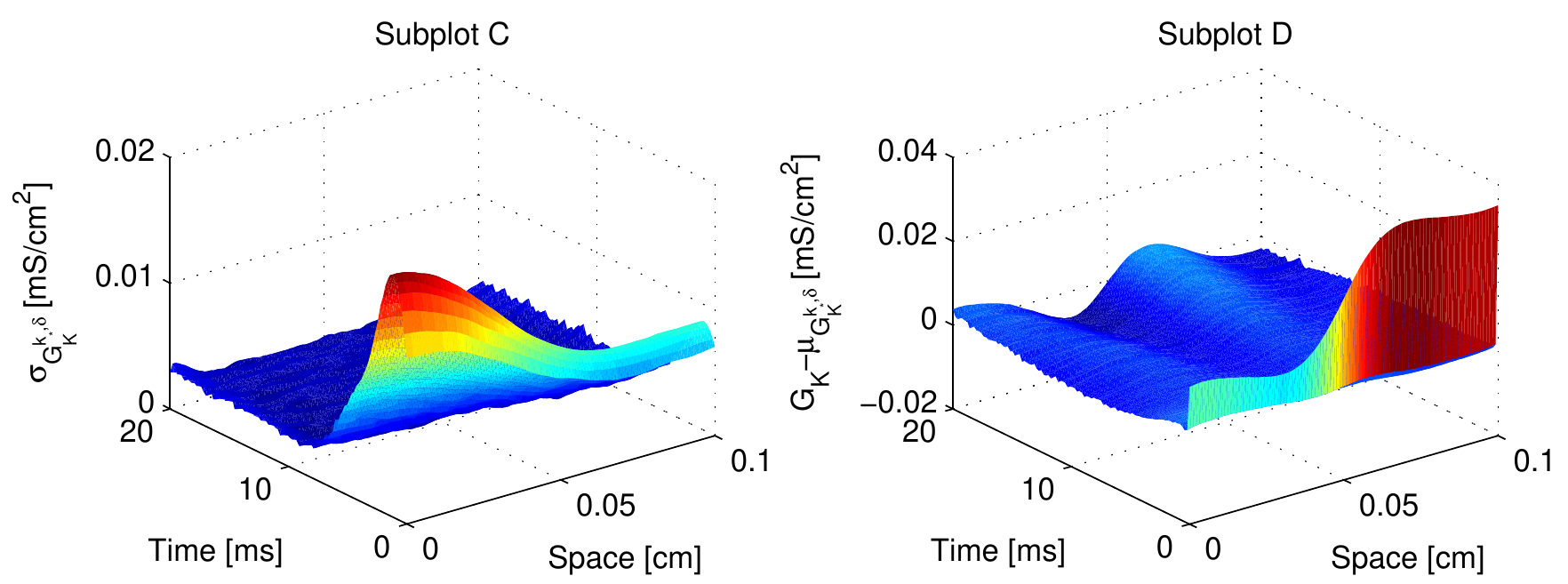}
	\caption{{\footnotesize Plots for Example~\ref{Exa3.2}. See Figure \ref{Ex1-fig2} for the  subplots description.} }
	\label{Ex2-fig2}
\end{figure}

\Example\label{Exa3.3}
Consider now two different ions, $\K$ and $\Na$, where $N_{\ion}=2$ $(\Ion=\{\K, \Na\})$ $E_\K=-12\;[mV]$ and $E_\Na=115\;[mV]$. The goal is to approximate
\[G_\K(x)=0.2+0.2/\left(\;1+\exp(\;(\;0.1/2-x\;)/0.01\;)\;\right)\]\;
 and 
\[\;G_{\Na}(x)=0.1+0.1/\left(\;1+\exp(\;(\;0.1/2-x\;)/0.01\;)\;\right),\]
given $V^\delta|_\Gamma=\big\{V^\delta(t,x):\,(t,x)\in [0,20]\times[0,0.1]\big\}$. 

The extra difficulty in this example lies in the fact that there are two conductance functions to be discovered. In Table~\ref{t:ex3} we present the results for various levels of noise. In Figures~\ref{Ex3-fig1},~\ref{Ex3-fig2} and \ref{Ex3-fig3}, we plot results for $\Delta=1\%$ of noise with $M=50$ experiments (see Table~\ref{t:ex2}, line 3). Note that now there are two conductances, one related to $\K$ and the other to  $\Na$. For results with  $\Delta=5\%$ see the Online Resource.
\begin{table}
	\small
	
\begin{tabular}{|l|l|l|c|c|c|c|c|c|c|}\hline
			$\Delta$  & $\Error_\bG$ & $\Error_V$  \\\hline  
$25\%$    & 8.2571 \%    & 1.7057 \%  \\\hline
$5\%$     & 2.1458 \%    & 2.1458 \%  \\\hline
$1\%$     & 0.5653 \%    & 0.0687 \%  \\\hline
$0.2\%$   & 0.2109 \%    & 0.0136 \%  \\\hline
\end{tabular}
\caption{\footnotesize Numerical results for Example~\ref{Exa3.3} with $M=50$ experiments for each noise level $\Delta$. The first column describes the noise level $\Delta$, as in  Eq.~\eqref{equ22}. The second column contains the mean errors according to Eq.~\eqref{equa21}. Finally, the third column contains the mean errors according to Eq. ~\eqref{equa23}.}
	\label{t:ex3}
\end{table}
\begin{figure}
	\centering
	\includegraphics[ height=5cm, width=14.cm]{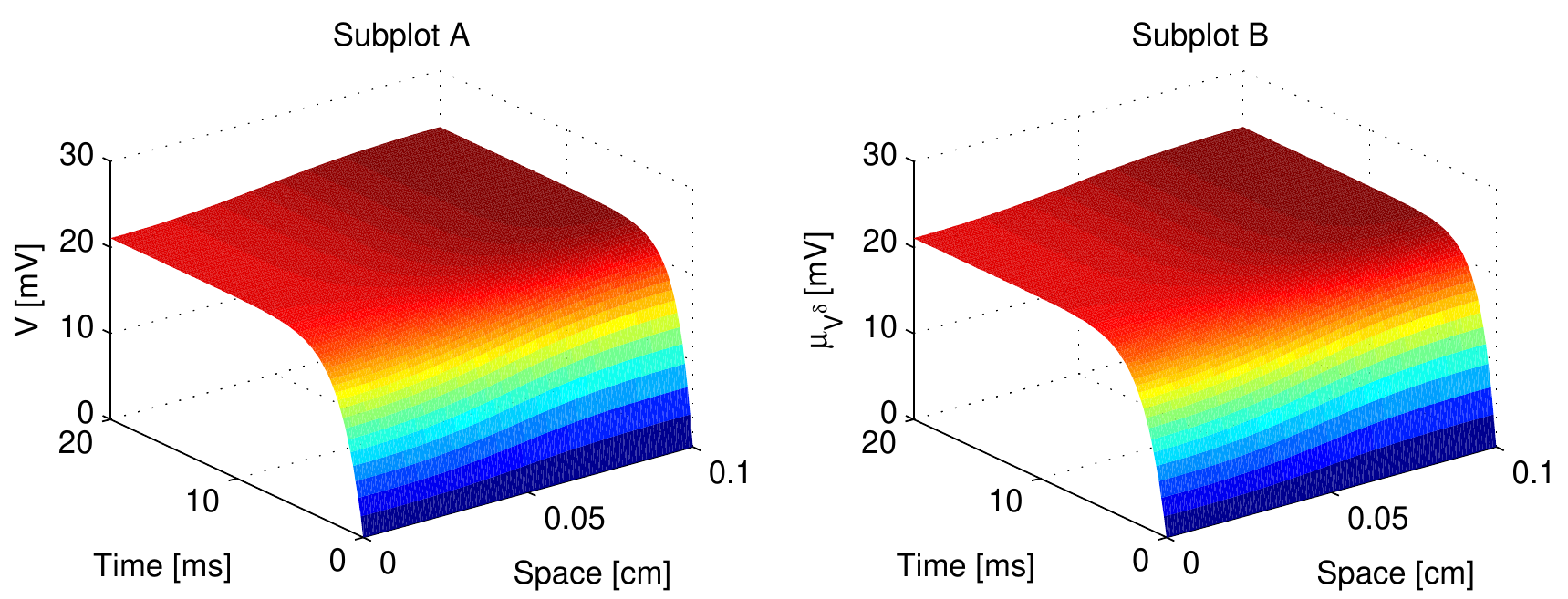}
	\includegraphics[ height=5cm, width=14.cm]{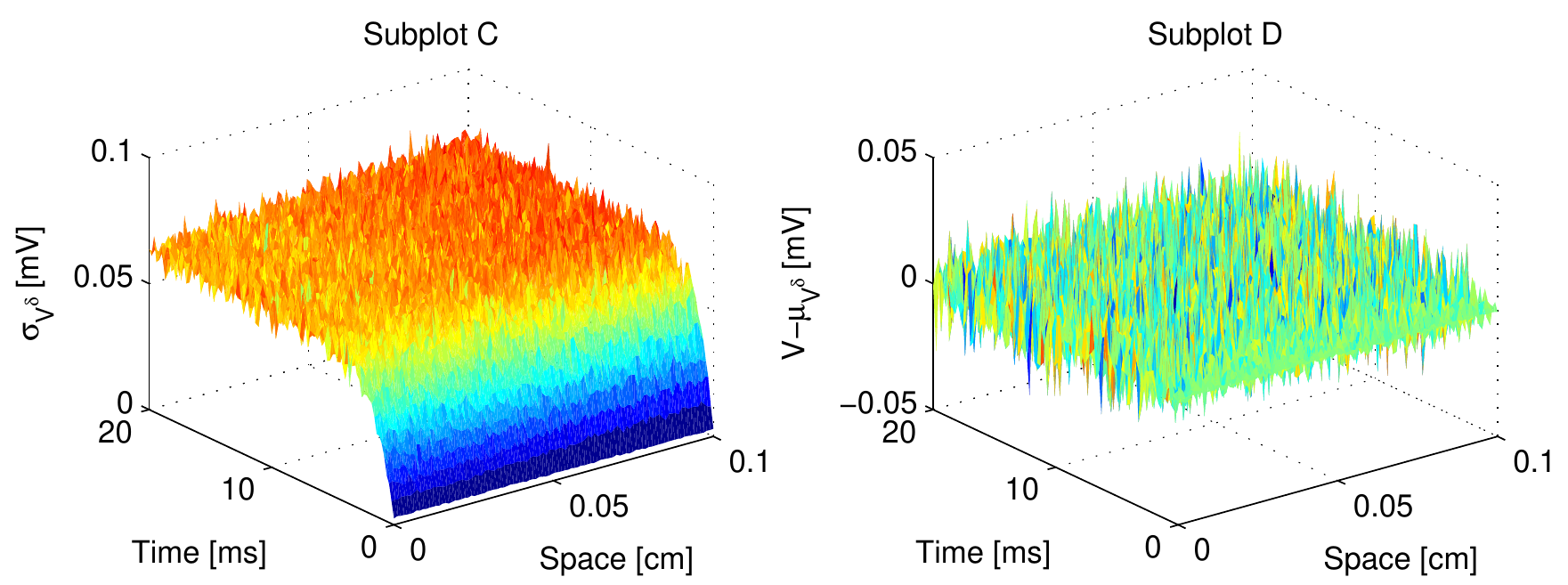}
	
	\caption{{\footnotesize For Example~\ref{Exa3.3} with  $\Delta=1\%$. See Figure \ref{Ex1-fig1} for the   subplots description.}}
	\label{Ex3-fig1}
\end{figure}   
\begin{figure}
	\centering
	\includegraphics[ height=5cm, width=14.cm]{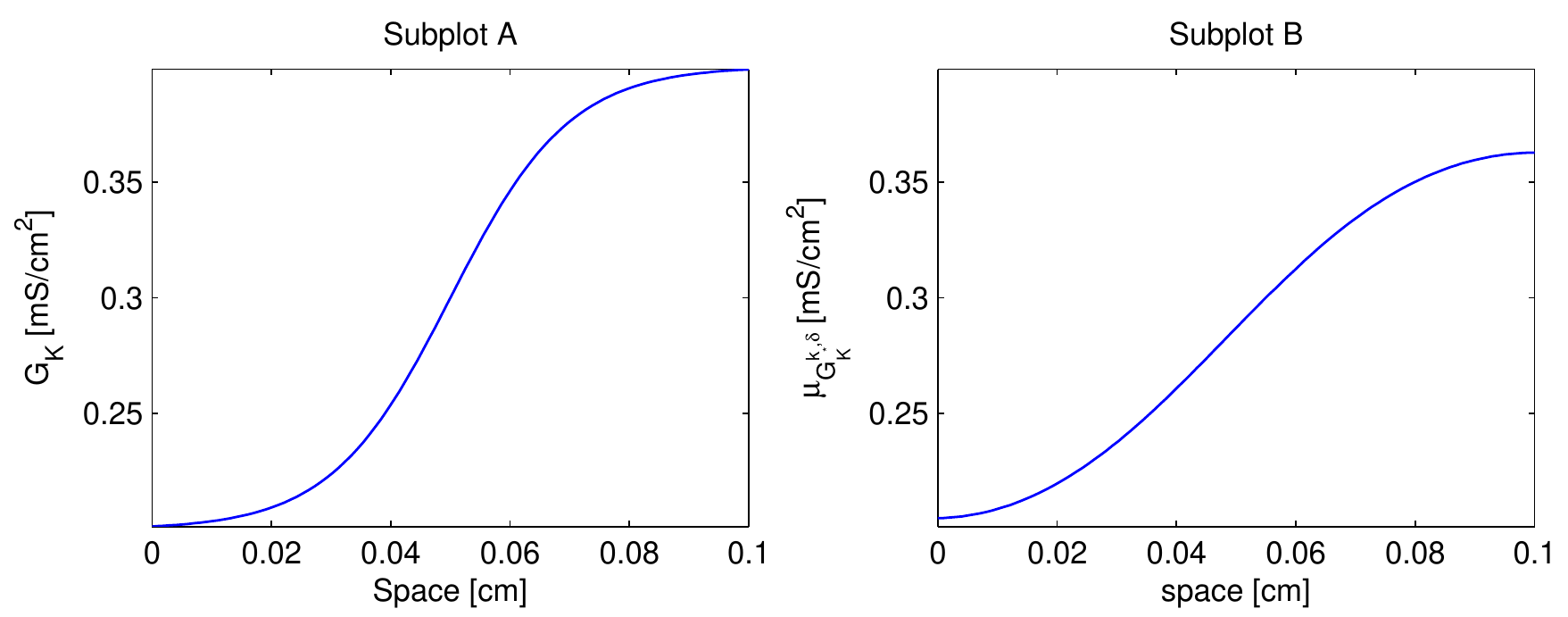}
	\includegraphics[ height=5cm, width=14.cm]{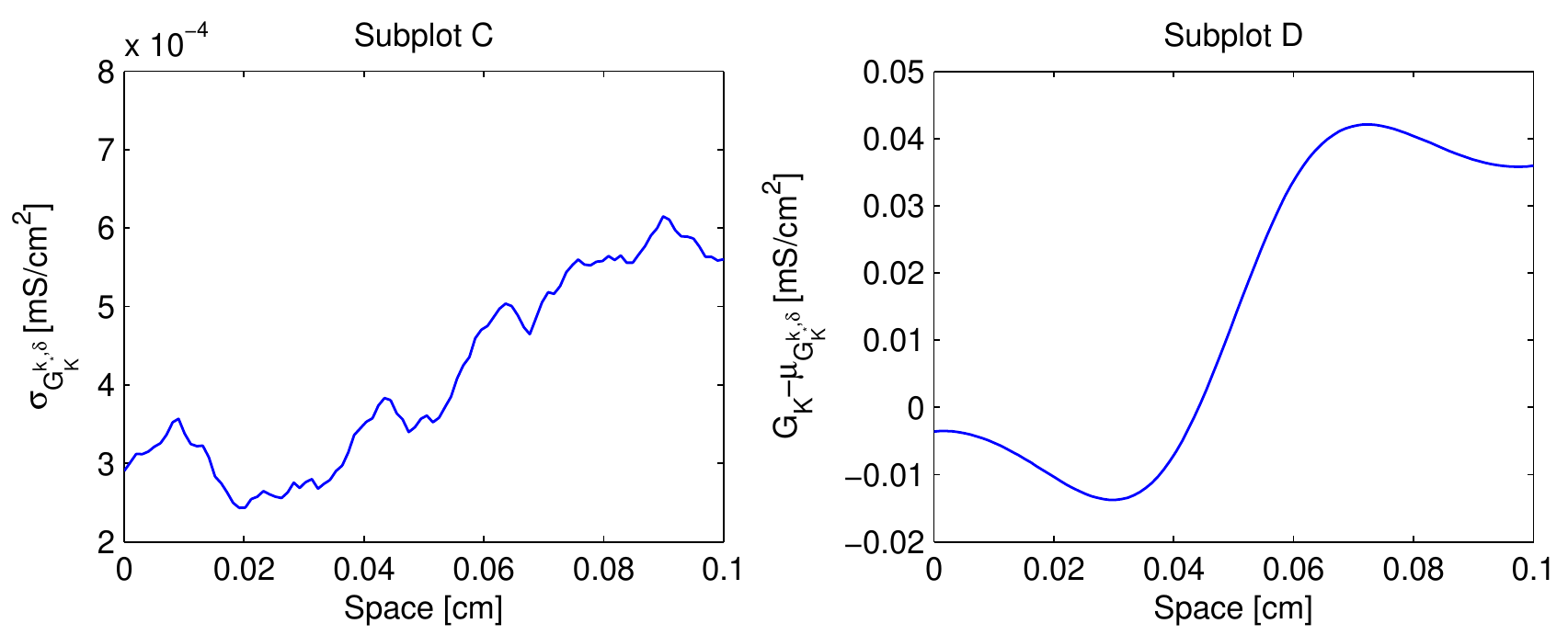}
	\caption{{\footnotesize Results for Example~\ref{Exa3.3} with  $\Delta=1\%$. See Figure \ref{Ex1-fig2} for the   subplots description..} }
	\label{Ex3-fig2}
\end{figure}
\begin{figure}
	\centering
	\includegraphics[ height=5cm, width=14.cm]{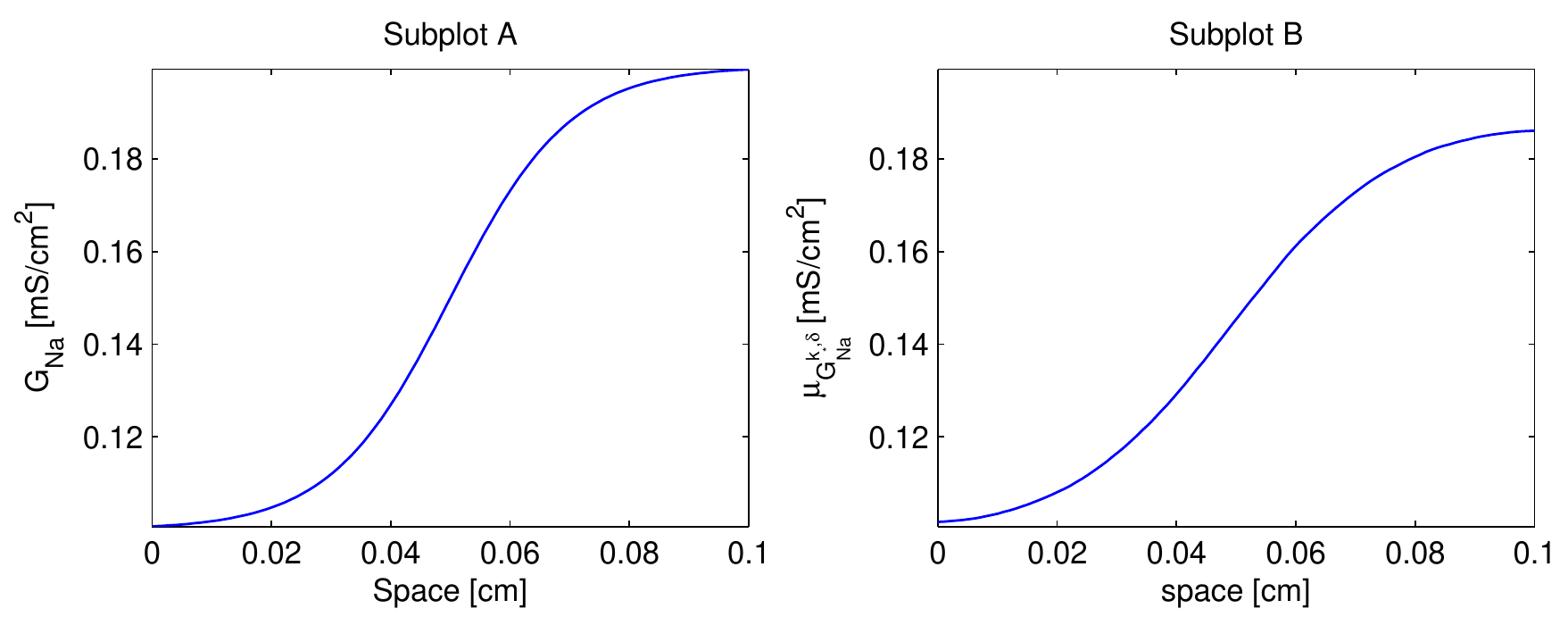}
	\includegraphics[ height=5cm, width=14.cm]{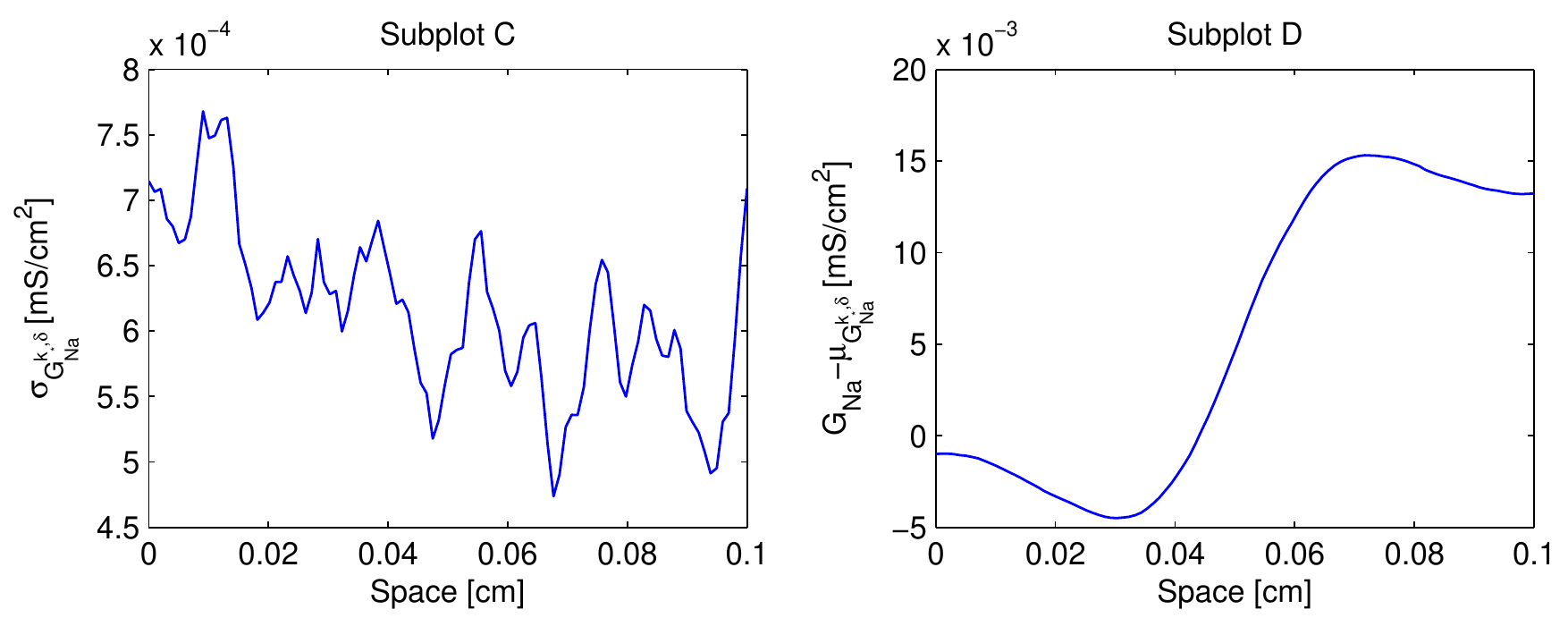}
	\caption{{\footnotesize Results for Example~\ref{Exa3.3} with $\Delta=1\%$ of noise. See Figure \ref{Ex1-fig2} for the subplots description.} }
	\label{Ex3-fig3}
\end{figure}

\Example\label{Exa3.4}
As our final example, we consider the domain defined by a tree, as discussed in Section~\ref{subs2.3}. We consider $N_\ion=1$  $(\Ion=\{\K\})$, $E_{\K}=-12 \;(mV)$ and $G_i(t,x)=G_{\K}(x)$. Given $V^\delta$ in $(0,T)\times\Theta$, the goal is to estimate
\[
G_K(x)=\begin{cases}
0.2+0.2/\left(\; 1+\exp(\; (\;0.1/2-\dist(x,\nu_1)\;)/0.01 \;) \;\right)&\text{if } x\in e_1,
\\
0.2+0.2/\left(\; 1+\exp(\; (\;0.1/2-0.01-\dist(x,\nu_2)\;)/0.01 \;) \;\right)&\text{if } x\in e_2\cup e_3,
\end{cases}
\]
where $\dist(a,b)$ denotes the distance between the points $a$ and $b$. Table~\ref{t:ex4} presents the results for various levels of noise. In figures~{\ref{Ex4-e1fig1}--\ref{Ex4-e3fig2}}, we plot the numerical result for $\Delta=1\%$, and in the Online Resource we plot the results for $\Delta=5\%$.
\begin{table}
	\small
	\begin{tabular}{|l|l|l|c|c|c|c|c|c|c|}\hline
		$\Delta$  & $\Error_\bG$ & $\Error_V$  \\\hline  
$25\%$    & 2.8572 \%    & 3.9757 \%   \\\hline
$5\%$     & 0.7661  \%   & 0.8127 \%   \\\hline
$1\%$     & 0.2926  \%   & 0.1605 \%   \\\hline
$0.2\%$   & 0.1204  \%   & 0.0319 \%   \\\hline
	\end{tabular}
	\caption{\footnotesize Numerical results for Example~\ref{Exa3.4} with $M=50$ experiments for each noise level $\Delta$. The first column describes the noise level $\Delta$, as in  Eq.~\eqref{equ22}. The second column contains the mean errors according to Eq.~\eqref{equa21}. Finally, the third column contains the mean errors according to Eq. ~\eqref{equa23}.}
	\label{t:ex4}
\end{table}
\begin{figure}
	\centering
	\includegraphics[ height=5cm, width=14.cm]{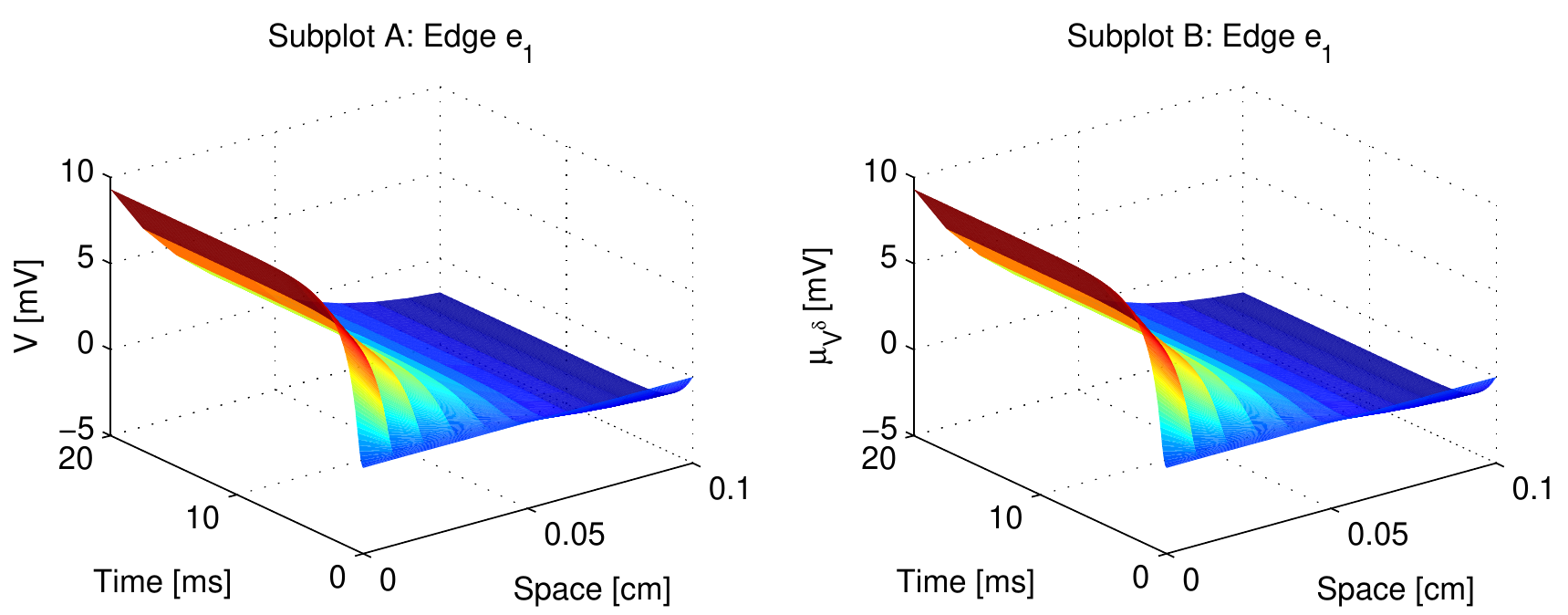}
	\includegraphics[ height=5cm, width=14.cm]{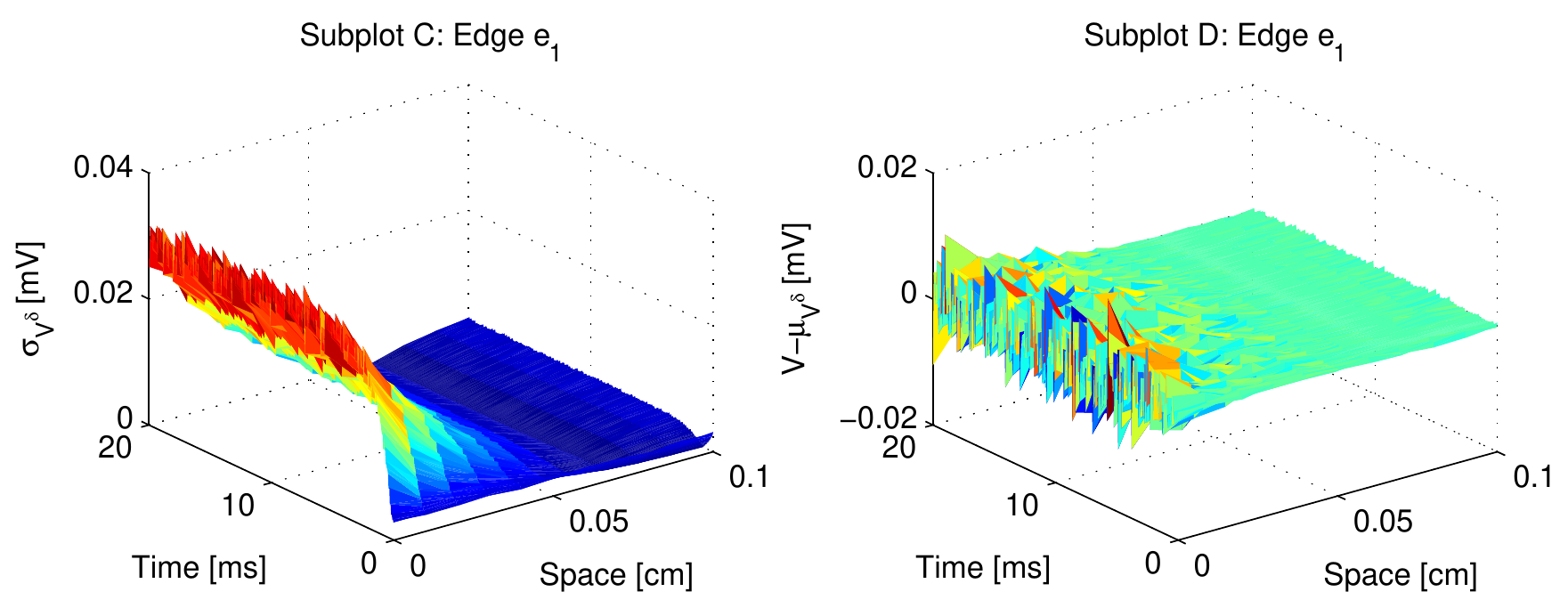}
	\caption{{\footnotesize For Example \ref{Exa3.4} and edge $e_1$, with  $\Delta=1\%$. See Figure \ref{Ex1-fig1} for the   subplots description.}}
	\label{Ex4-e1fig1}
\end{figure}   
\begin{figure}
	\centering
	\includegraphics[ height=5cm, width=14.cm]{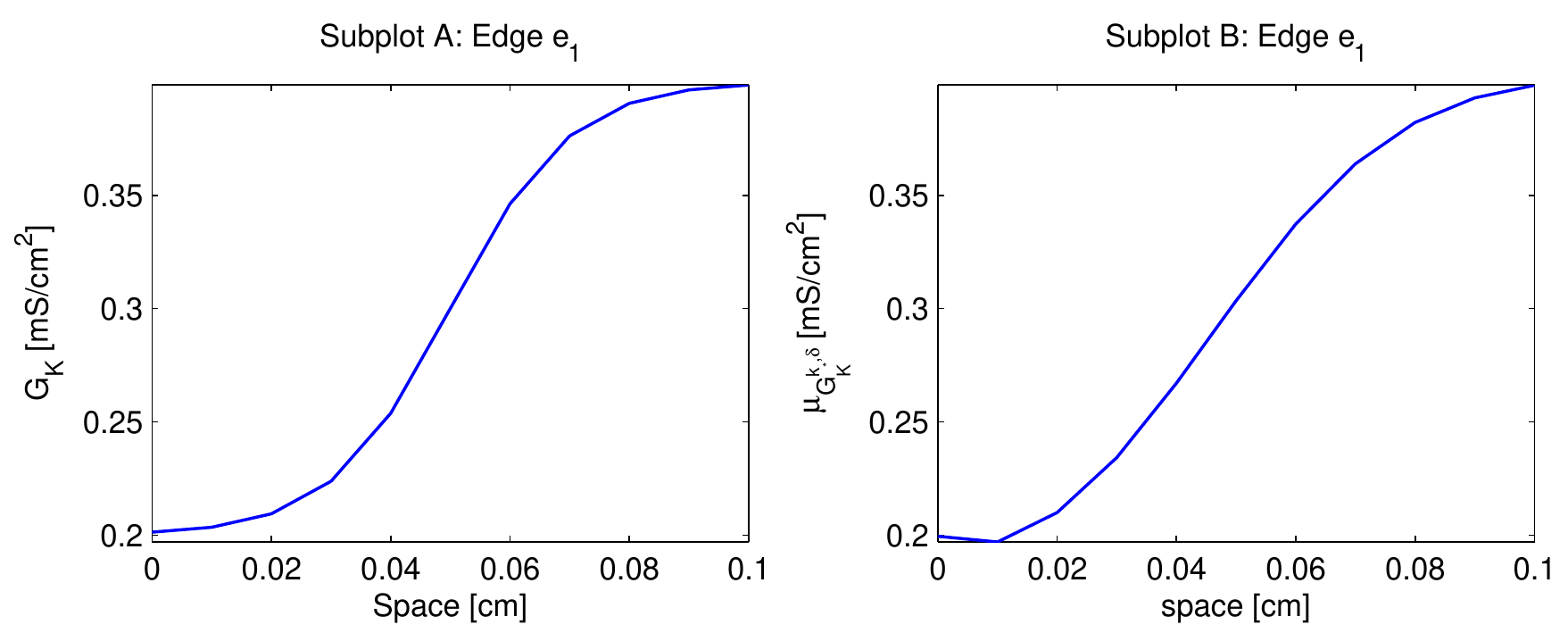}
	\includegraphics[ height=5cm, width=14.cm]{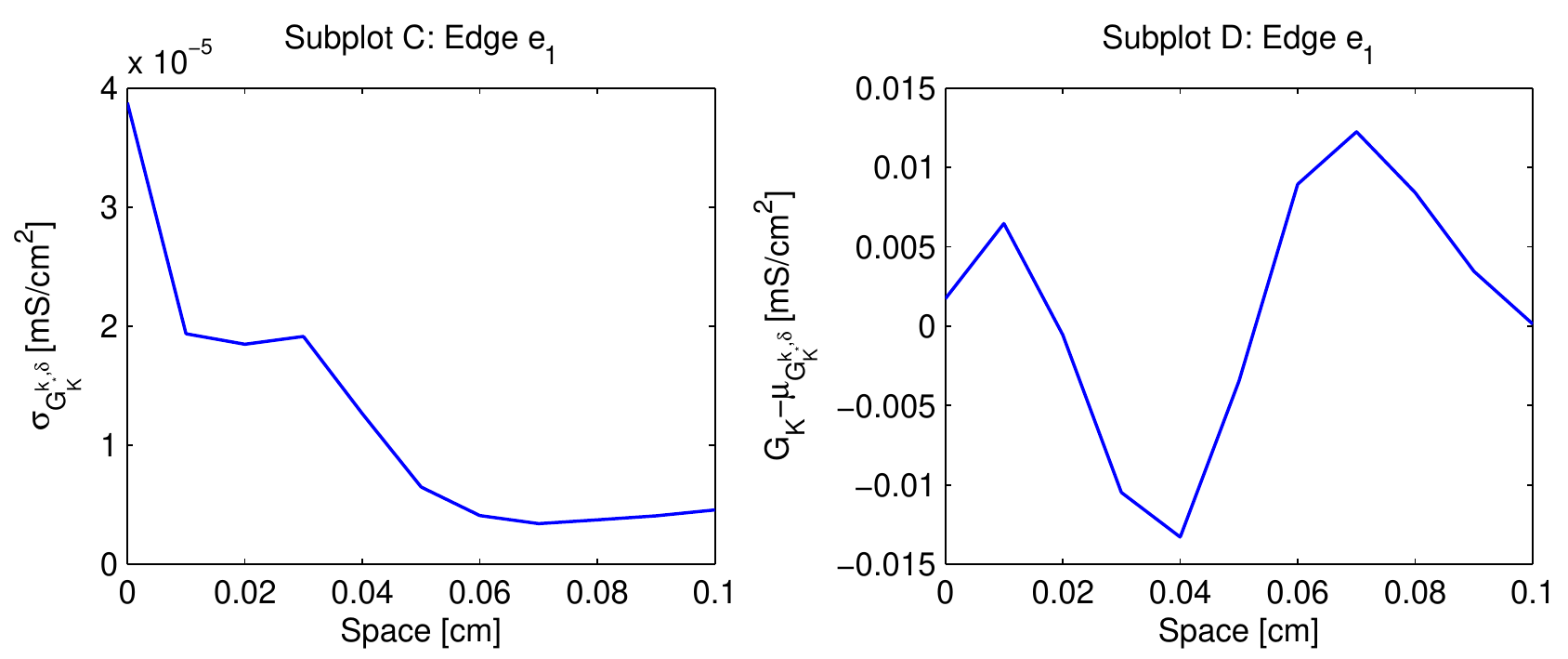}
	\caption{{\footnotesize Plots for Example~\ref{Exa3.4} and edge $e_1$}. See Figure \ref{Ex1-fig2} for the  subplots description.} 
	\label{Ex4-e1fig2}
\end{figure}

\begin{figure}
	\centering
	\includegraphics[ height=5cm, width=14.cm]{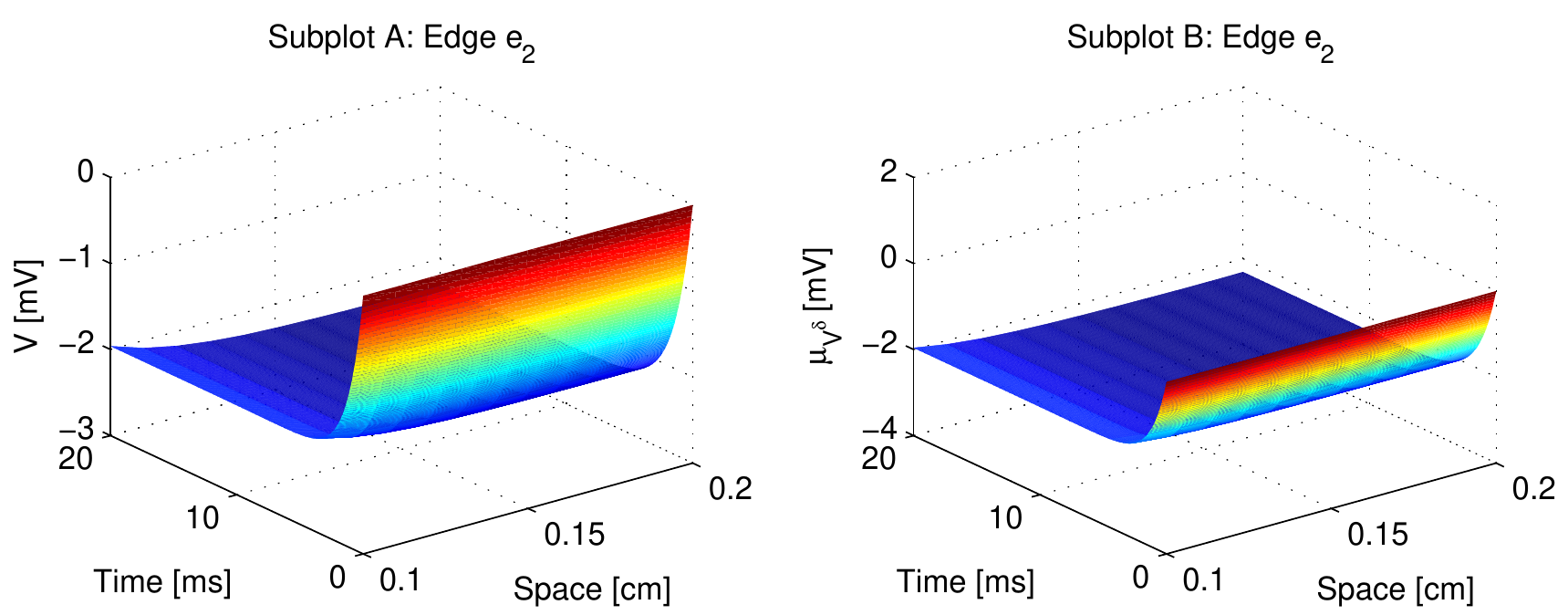}
	\includegraphics[ height=5cm, width=14.cm]{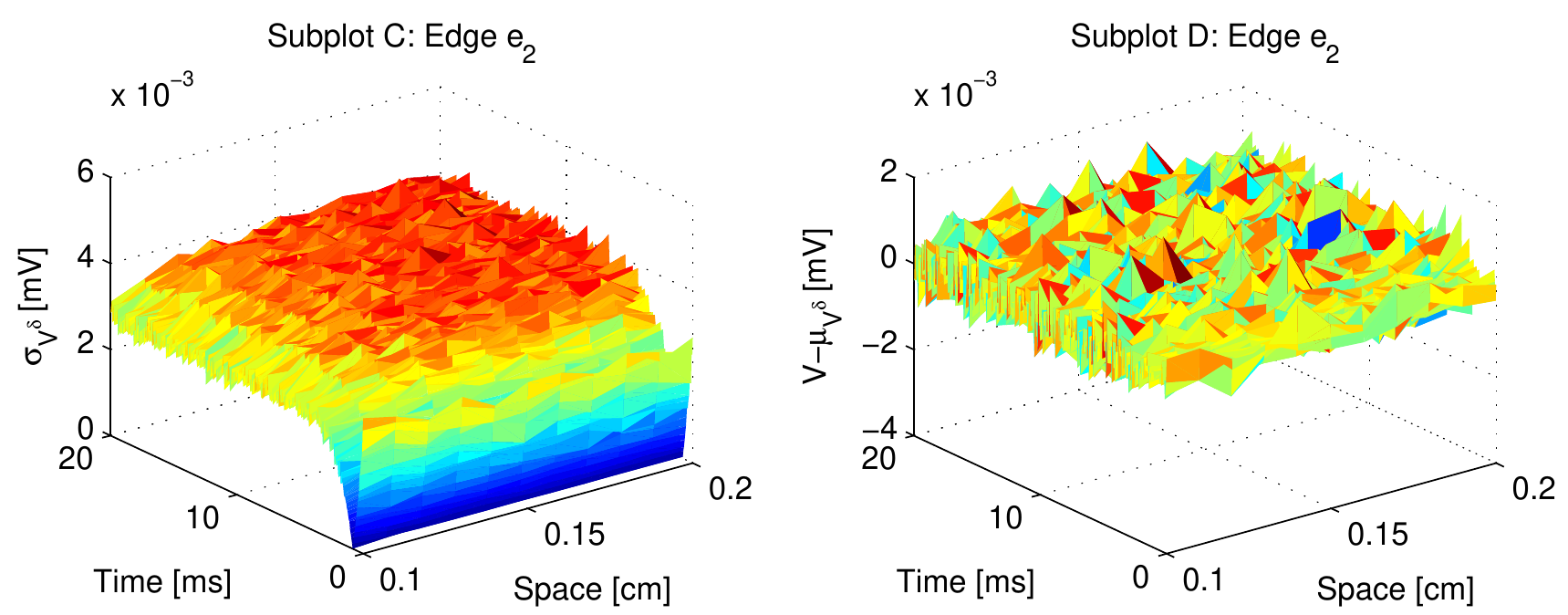}
	\caption{{\footnotesize For Example~\ref{Exa3.4} and edge $e_2$}, with $\Delta=1\%$. See Figure~\ref{Ex1-fig1} for the subplots description.}
	\label{Ex4-e2fig1}
\end{figure}   
\begin{figure}
	\centering
	\includegraphics[ height=5cm, width=14.cm]{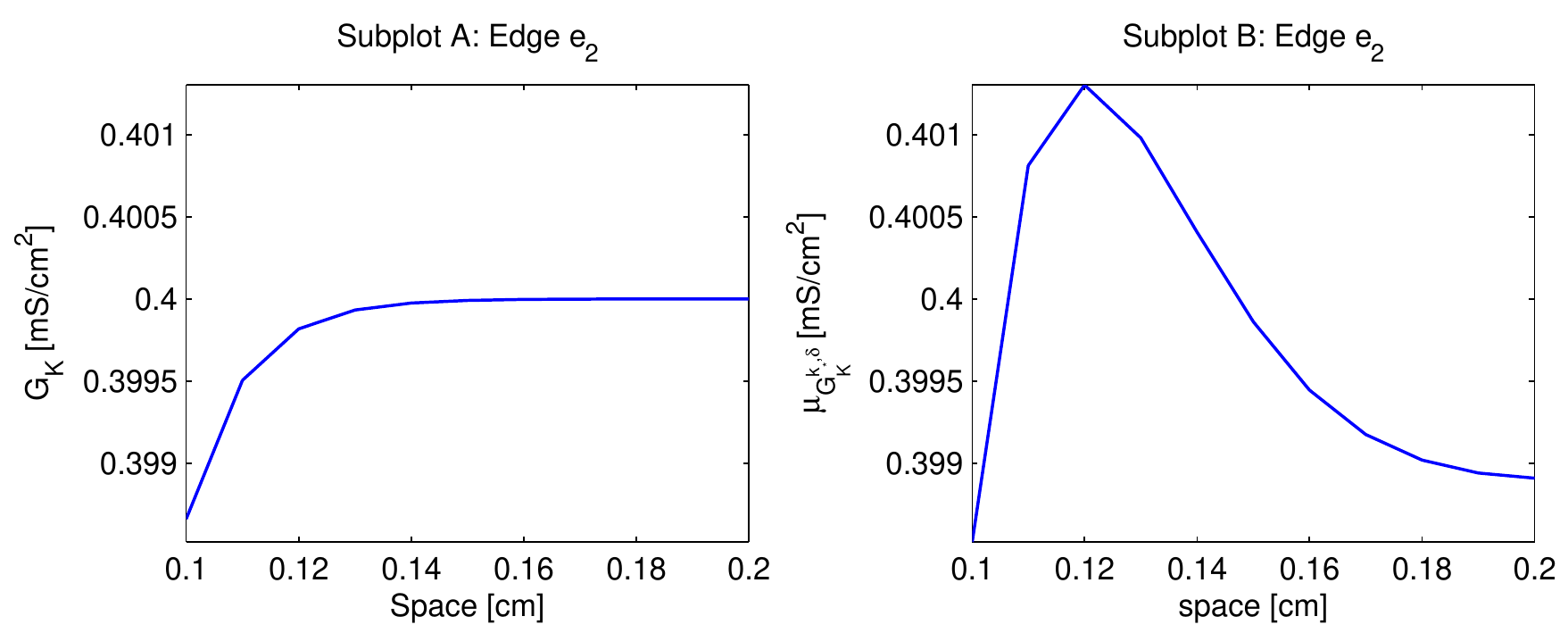}
	\includegraphics[ height=5cm, width=14.cm]{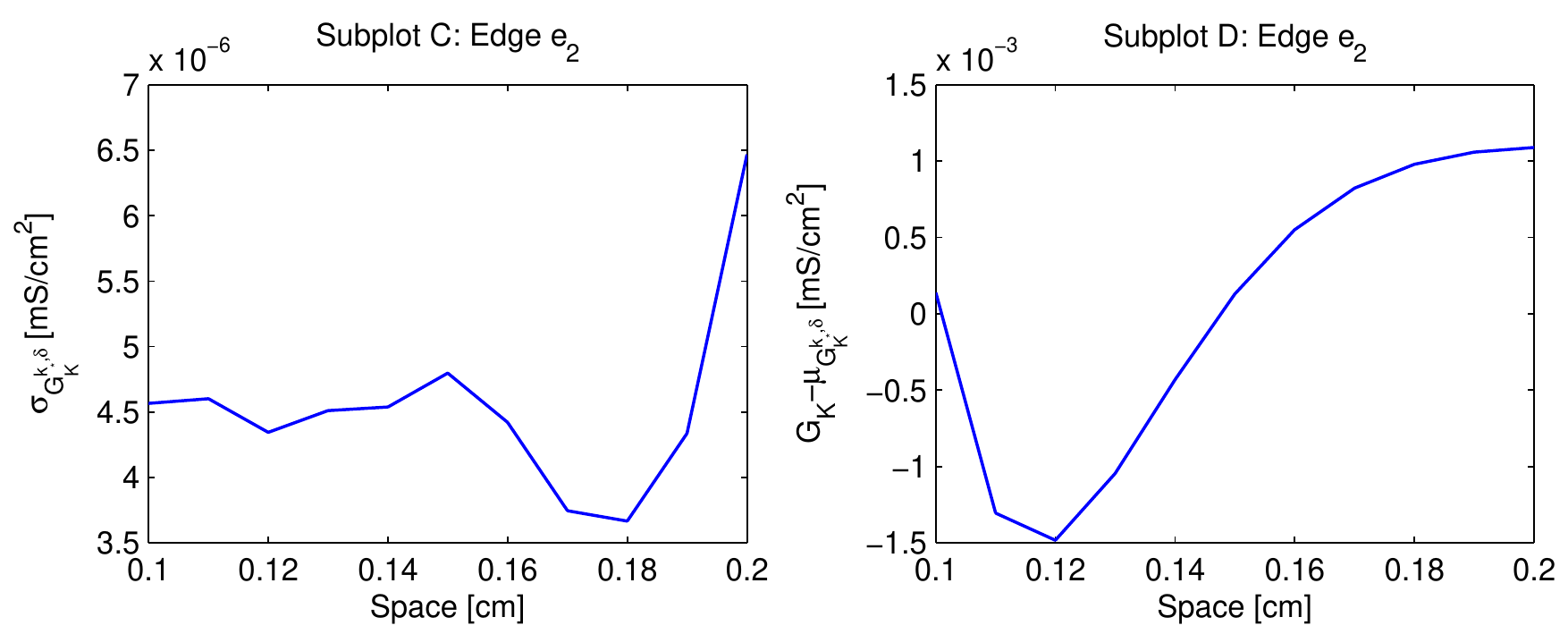}
	\caption{{\footnotesize Plots for Example~\ref{Exa3.4} and edge $e_2$}. See Figure \ref{Ex1-fig2} for the  subplots description.} 
	\label{Ex4-e2fig2}
\end{figure}

\begin{figure}
	\centering
	\includegraphics[ height=5cm, width=14.cm]{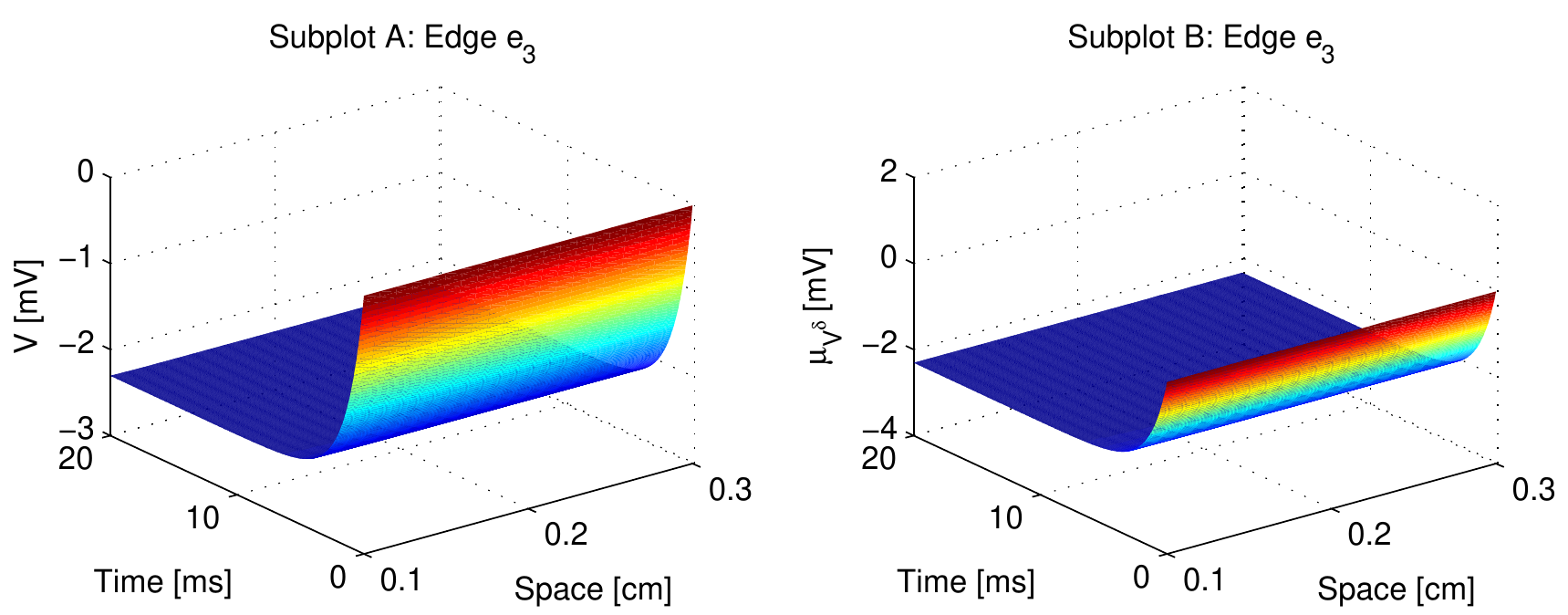}
	\includegraphics[ height=5cm, width=14.cm]{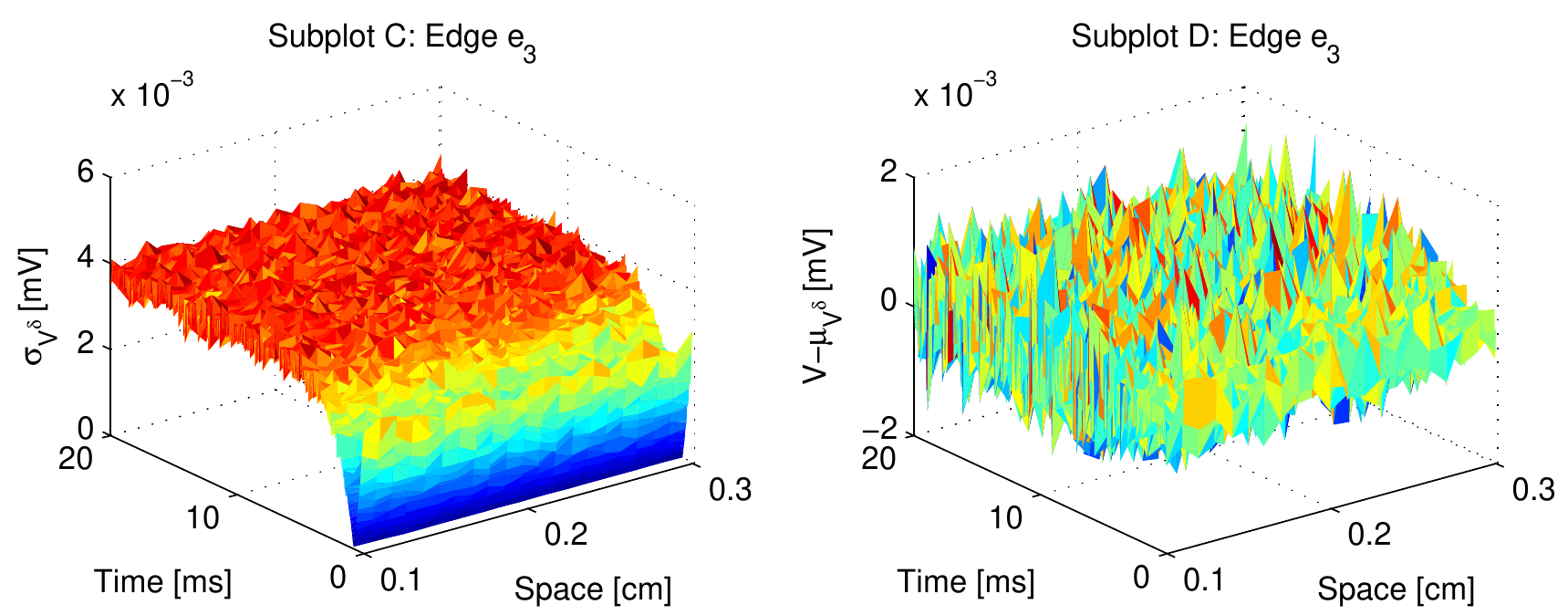}
	\caption{{\footnotesize For Example \ref{Exa3.4} and edge $e_3$}, with $\Delta=1\%$. See Figure~\ref{Ex1-fig1} for the subplots description.}
	\label{Ex4-e3fig1}
\end{figure}   
\begin{figure}
	\centering
	\includegraphics[ height=5cm, width=14.cm]{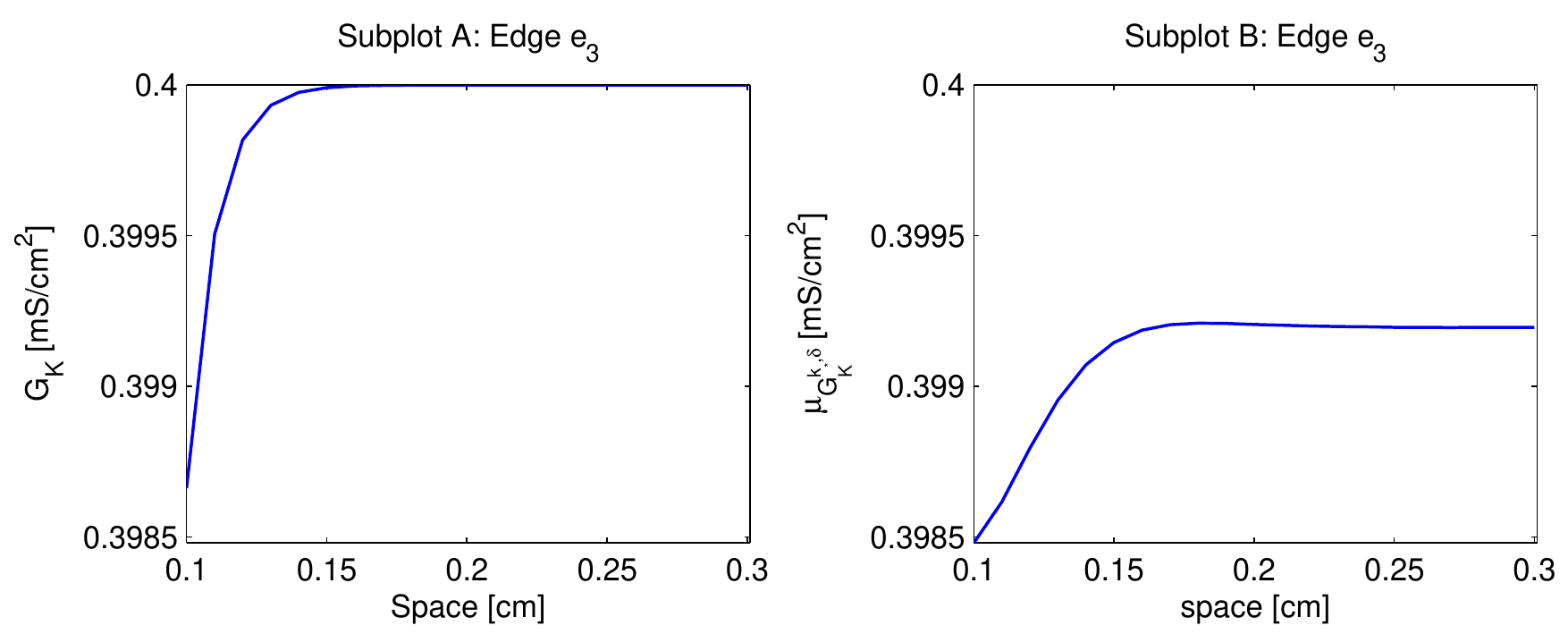}
	\includegraphics[ height=5cm, width=14.cm]{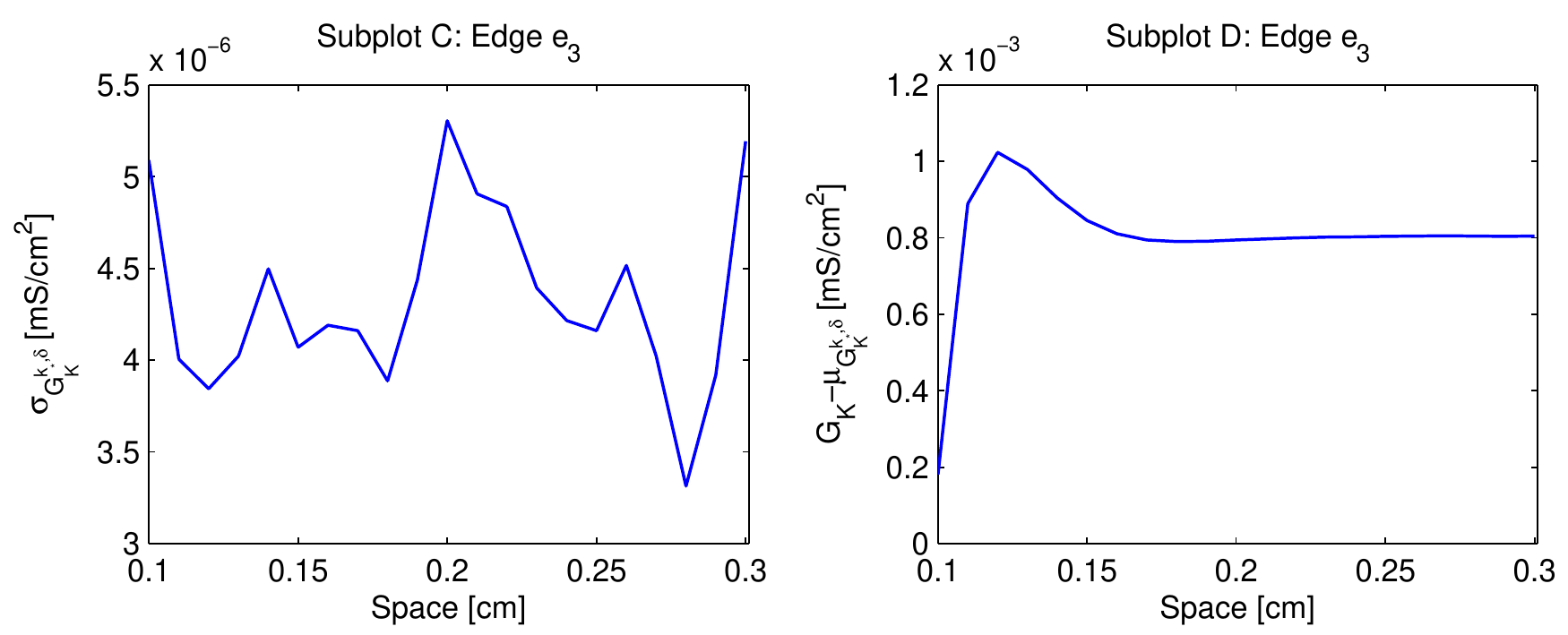}
	\caption{{\footnotesize Plots for Example~\ref{Exa3.4} and edge $e_3$}. See Figure~\ref{Ex1-fig2} for the  subplots description.} 
	\label{Ex4-e3fig2}
\end{figure}

\section{Conclusions}\label{s:conc}
The inverse problem of finding conductances from voltage data for a given neuron model is important and difficult. This paper presents and tests a way to approximate them based on the minimal error iterative method and applied to a cable model. Although the scheme has a somewhat straightforward description, it is not practical in its original formulation since computing the adjoint of the G\^ateux derivative seems unfeasible in general. The development of auxiliary equations to overcome such hurdle requires some art, and is done on a case-by-case basis.

Certainly, the method has limitations and is no panacea. How well the method performs depends on the noise, on how close to the solution is the initial guess, the amount of data, and on the model used. There is a nontrivial interplay between all those conditions. For instance, determining two conductances is harder than determining one, finding conductances that depend on time and space is harder than finding conductances that depend on space only. Also, having data at all points is better than if the data is available at isolated points only.

Our examples display some of these features. For some of them, the method performs nicely, capturing the correct conductances. If the level of noise increases, the method delivers reasonable approximations (see our Online Resource for that), but these approximations cannot be qualitatively better than the available data, especially for inverse problems. Inverse problems are unstable, and thus, not well-posed and difficult to solve in general. Even when the method does not do a good job in capturing the correct conductance, the computed residual is small, and whenever the residual is of the same order as the noise, there is no point in iterating any further. 

Under reasonable conditions, the method yields good results even in the presence of noise, as shown here. It is also general enough to accommodate for different geometries (straight cables and branched trees), and different measured data (endpoints, whole cable). 

We believe that methods that are capable of inferring spatial properties of neurons are in demand and will grow in importance, in particular due to new imagining techniques such as VSDI. Also, regularizing methods for inverse problems are applied in several research fields and they can also contribute to Neuroscience.

Finally, we note that inverse problems are nonlinear by nature, even if the equations involved are linear. If the equations are nonlinear, then the problem becomes much harder since finding the adjoint of the Gateaux derivatives of the operators involved is highly nontrivial. The present work represents a first step towards finding conductances for neuronal models. The final goal is to develop a similar method for the Hodgkin-Huxley system, but the highly nonlinear nature of those equations makes the task daunting, even when considering ODEs.

\appendix
\section{Abstract Formulation}\label{Apendix}
In practice, ${V|}_\Gamma$ is  the data and given such information and under the assumption that Eq.~\eqref{equ3} holds, the inverse problem under consideration is to recover or approximate the conductances. The lack of stability, characteristic of ill-posed problems can be tamed by regularization methods~\cite{engl1996,kaltenbacher2008,kirsch2011}, in particular by the  minimal error method.

Consider for simplicity $T>0$. Let $\Omega=\{(t,x):\, 0\leq t\leq T, 0\leq x \leq L\}$, and
\begin{gather*}
H(F)=\bigl(L^2(\Omega)\bigr)^{N_\ion}
=\biggl\{f:\Omega\rightarrow\R^{N_\ion}:\,\int_\Omega|f(\xi)|^2d\,\xi<\infty \biggr\}, 
\\
R(F)=L^2(\Gamma)=\biggl\{f:\Gamma \rightarrow \R:\,\int_\Gamma|f(\xi)|^2d\,\xi<\infty\biggr\}.
\end{gather*}

It is well-known that $H(F)$ and $R(F)$ become Hilbert spaces under the inner products 
\begin{equation*}
\begin{gathered}
{\langle f,h\rangle_{H(F)}=\int_{\Omega}f(\xi)h(\xi)d\xi,}
\qquad
{\langle f,h\rangle_{R(F)}=\int_{\Gamma}f(\xi)h(\xi)d\xi,}. 
\end{gathered}
\end{equation*}
and the associated norms $\|f\|_{H(F)}=\langle f,f\rangle_{H(F)}^{1/2}$, $\|f\|_{R(F)}=\langle f,f\rangle_{R(F)}^{1/2}$. Note that the inner product on $R(F)$ depends on $\Gamma$, see Eq.~\eqref{e:alphadef}, as follows:

\begin{multline}\label{equ5}
\langle f,h\rangle_{R(F)}=\alpha_1\int_0^L\int_0^Tf(t,x)h(t,x)\,dt\,dx
                    +\alpha_2\int_0^Tf(t,0)h(t,0)dt\\
                    +\alpha_2\int_0^Tf(t,L)g(t,L)\,dt,
\end{multline}
where $\alpha_1$, $\alpha_2$ are as in Eq.~\eqref{e:alphadef}.

The set  $D(F)=\left(L^{\infty}(\Omega)\right)^{N_\ion}\subset H(F)$ is the Banach space of ``essentially" bounded functions (see~\cite{kreyszig1978} for precise definitions). Consider the operator $F:D(F)\subset H(F)\rightarrow R(F)$ defined by $F(\bG)=V|_{\Gamma}$. Our goal is to find an approximation for $\bG$ using the minimal error iteration defined by Eq.~\eqref{equ6}.

In the next Theorem we show how to  obtain Eq.~\eqref{equ16} from Eq.~\eqref{equ6}.

\begin{thm}\label{t:main}
Consider the  iteration in Eq.~\eqref{equ6}. Then Eq.~\eqref{equ16} holds
\end{thm}

\begin{proof}
Given $\bG^{k,\delta}\in D(F)$ and $\btheta=(\theta_1,\dots,\theta_{N_\ion})\in H(F)$, the G\^ateux derivative of $F$ at $\bG^{k,\delta}$ in the direction $\btheta$ is given by 
\begin{equation}\label{equ8}
F'(\bG^{k,\delta})(\btheta)
=\lim_{\lambda\to0}\frac{F(\bG^{k,\delta}+\lambda\btheta)-F(\bG^{k,\delta})}{\lambda}
=W^k|_\Gamma, 
\end{equation}
where $W^k$ solves 
\begin{equation}\label{equ9}
\left \{\begin{array}{l}
\displaystyle\frac{r_a}{2R}W_{xx}^k(t,x)-C_MW_t^k(t,x)-G_LW^k(t,x)
\vspace*{0.1cm}\\\hspace*{4cm}\displaystyle-\sum_{i\in\Ion}G_i^{k,\delta}(t,x)W^k(t,x)
=\sum_{i\in\Ion}\theta_i(V^{k,\delta}(t,x)-E_i)\quad\text{in }\Omega, 
\vspace*{0.2cm}\\
W^k(0,x)=0\quad\text{for }\;0<x<L,
\qquad 
\vspace*{0.2cm}\\
W^k_x(t,0)=W^k_x(t,L)=0\quad\text{for }\;0<t<T, 
\end{array} \right.
\end{equation}
and $V^{k,\delta}$ solves Eq.~\eqref{equ3} with $G_i$ replaced by $G_i^{k,\delta}$. To obtain Eq.~\eqref{equ9} from Eq.~\eqref{equ8}, it is enough to consider the difference between problem in Eq.~\eqref{equ3} with coefficients $\bG^{k,\delta}+\lambda\btheta$ and $\bG^{k,\delta}$, divide by $\lambda$ and take the limit $\lambda\to0$. 

Let $V^{k,\delta}|_\Gamma=F(\bG^{k,\delta})$. From the minimal error iteration in  Eq.~\eqref{equ6}, we gather that
\begin{eqnarray*}
\la\bG^{k+1,\delta}-\bG^{k,\delta},\btheta\ra_{H(F)}&=&w^{k,\delta}\la F'(\bG^{k,\delta})^*(V^\delta|_\Gamma-F(\bG^{k,\delta})),\btheta\ra_{H(F)}\\
&=&w^{k,\delta}\la F'(\bG^{k,\delta})^*(V^\delta|_\Gamma-V^{k,\delta}|_\Gamma),\btheta\ra_{H(F)}. 
\end{eqnarray*}

By definition of adjoint operator, 
\begin{multline}\label{equ11}
\la\bG^{k+1,\delta}-\bG^{k,\delta},\btheta\ra_{H(F)}=w^{k,\delta}\la V^\delta|_\Gamma-V^{k,\delta}|_\Gamma,F'(\bG^{k,\delta})(\btheta)\rangle_{R(F)}\\
=w^{k,\delta}\la V^\delta|_\Gamma-V^{k,\delta}|_\Gamma,W^k|_\Gamma \ra_{R(F)},
\end{multline}
from Eq.~\eqref{equ8}. 

Although Eq.~\eqref{equ11} yields an interesting relation, it carries an impeding dependence on $\btheta$ through $W^k$. It is possible to avoid that by performing some ``trick'' manipulations.

 Multiplying the first equation from~\eqref{equ10} by $-W^k$, and integrating in the intervals $[0,T]$ and $[0,L]$ we gather that 
\begin{multline}\label{equ12}
\int_0^L\int_0^T\frac{r_a}{2R} U^k_{xx}(t,x)W^k(t,x)\,dt\,dx+\int_0^L\int_0^TC_M\;U^k_t(t,x)W^k(t,x)\,dt\,dx
\\
-\int_0^L\int_0^TG_LU^k(t,x)W^k(t,x)dtdx-\int_0^L\int_0^T\sum_{i \in ion}G_i^{k,\delta}(t,x)\;U^k(t,x)W^k(t,x)\,dt\,dx=
\\
-\alpha_1\int_0^L\int_0^T \left(V^\delta(t,x)-V^{k,\delta}(t,x)\right) W^k(t,x)\,dt\,dx. 
\end{multline}

Integrating by parts twice the first term from Eq.~\eqref{equ12} with respect to the space variable, and using the boundary conditions for $W^k$ we have
\begin{multline}\label{equ13}
\int_0^L\int_0^T\frac{r_a}{2R} U^k_{xx}(t,x)W^k(t,x)\,dt\,dx=\int_0^L\int_0^T\frac{r_a}{2R}U^k(t,x)W_{xx}^k(t,x)\,dt\,dx\\
+\frac{r_a}{2R}\int_0^T U^k_x(t,x)W^k(t,x)|_0^L\,dt, 
\end{multline}
where we denote $U^k_x(t,x)W^k(t,x)|_0^L=U^k(t,L)W^k(t,L)-U^k(t,0)W^k(t,0)$. 
Similarly, integrating by parts the second term from Eq.~\eqref{equ12} with respect to time and using the initial condition of $W^k$ and the final condition of $U^k$, we gather that 
\begin{equation}\label{equ14}
\int_0^L\int_0^TC_M\;U^k_{t}(t,x)W^k(t,x)\,dt\,dx=-\int_0^L\int_0^TC_MU^k(t,x)W_t^k(t,x)\,dt\,dx. 
\end{equation}

Substituting Eqs.~\eqref{equ13} and ~\eqref{equ14} in Eq.~\eqref{equ12}, it follows that 
\begin{multline*}
\int_0^L\int_0^T\Biggl(\frac{r_a}{2R} W_{xx}^k(t,x)-C_MW_t^k(t,x)-G_LW^k(t,x)-\sum_{i \in ion}G_i^{k,\delta}(t,x)W^k(t,x) \Biggr)U^k(t,x)\,dt dx
\\
=-\alpha_1\int_0^L\int_0^T \bigl(V^\delta(t,x)-V^{k,\delta}(t,x)\bigr)W^k(t,x)\,dt\,dx-\frac{r_a}{2R}\int_0^TU^k_x(t,x)W^k(t,x)|_0^L\,dt. 
\end{multline*}

Substituting the first equation from \eqref{equ9} in the previous equation, we obtain 
\begin{multline*}
\int_0^L\int_0^T\sum_{i \in \Ion}\theta_i(V^{k,\delta}(t,x)-E_i) U^k(t,x) \,dt \,dx \\
=-\alpha_1\int_0^L\int_0^T \left(V^\delta(t,x)-V^{k,\delta}(t,x)\right) W^k(t,x)\,dt dx
-\frac{r_a}{2R}\int_0^TU^k_x(t,x)W^k(t,x)|_0^L\,dt. 
\end{multline*}

From the boundary conditions from Eq.~\eqref{equ10}, the following expression holds:
\begin{multline*}
\int_0^L\int_0^T\sum_{i \in \Ion}\theta_i(V^{k,\delta}(t,x)-E_i) U^k(t,x) \,dt dx=\\
-\alpha_1\int_0^L\int_0^T \left(V^\delta(t,x)-V^{k,\delta}(t,x)\right) W^k(t,x)\,dt dx
-\alpha_2\int_0^T\left(V^\delta(t,0)-V^{k,\delta}(t,0)\right)W^k(t,0)\\
-\alpha_2\int_0^T\left(V^\delta(t,L)-V^{k,\delta}(t,L)\right)W^k(t,L)\,dt.
\end{multline*}

From the previous equation and the definition of the inner product in Eq.~\eqref{equ5}, we have
\begin{equation}\label{equ15}
\int_0^L\int_0^T\sum_{i \in \Ion}\theta_i(V^{k,\delta}(t,x)-E_i) U^k(t,x) \,dt dx =-\la V^\delta|_\Gamma-V^{k,\delta}|_\Gamma,W^k|_\Gamma \ra _{R(F)}.
\end{equation}

From Eqs.~\eqref{equ11} and~\eqref{equ15} we have 
\begin{multline*}
\int_0^L\int_0^T\sum_{i\in\Ion}\theta_i\left(  G_i^{k+1,\delta}(t,x)-G_i^{k,\delta}(t,x)\right)\,dt dx 
\\
=-w^{k,\delta}\int_0^L\int_0^T\sum_{i\in\Ion}\theta_i(V^{k,\delta}(t,x)-E_i) U^k(t,x) \,dt dx.
\end{multline*}

Since $\btheta\in H(F)$ is arbitrary and $L^\infty(\Omega)$ is dense in $L^2(\Omega)$, we gather that the following  iteration holds: 
\[G_i^{k+1,\delta}(t,x)=G_i^{k,\delta}(t,x)-w^{k,\delta}(V^{k,\delta}(t,x)-E_i)U^k(t,x)\quad\text{for all }i\in\Ion.
\]

\end{proof}

\bibliographystyle{acm}
\bibliography{mybibfile}

\end{document}